\documentclass[11pt]{article}

\usepackage{odonnell}
\usepackage{enumitem}
\usepackage{geometry}

\usepackage{booktabs}

\newcommand{\shivam}[1]{{\color{violet} \textbf{Shivam:} #1}}
\newcommand{\abra}[1]{{\left\langle #1 \right\rangle}}
\newcommand{\pbra}[1]{{\left( #1 \right)}}
\newcommand{\cbra}[1]{{\left\{ #1 \right\}}}
\newcommand{\sbra}[1]{{\left[ #1 \right]}}
\newcommand{\dom}[1]{\mathrm{Dom}\left( #1 \right)}

\newcommand{\Fcsc}{\calF_\mathrm{csc}}
\newcommand{\Fmon}{\calF_\mathrm{mon}}
\newcommand{\Fcsq}{\calF_\mathrm{csq}}
\newcommand{\Fcvx}{\calF_\mathrm{cvx}}

\newcommand{\operatorP}{\mathrm{P}}
\renewcommand{\P}{\mathrm{P}}
\renewcommand{\L}{\mathrm{L}}

\renewcommand{\Inf}{\mathbf{Inf}} 



\newcommand{\snote}[1]{\footnote{\color{violet}Shivam: {#1}}}

\newcommand{\rnote}[1]{\footnote{\color{red}Rocco: {#1}}}

\newcommand{\sign}{\mathrm{sign}}

\newcommand{\U}{\mathrm{U}}
\newcommand{\T}{\mathrm{T}}

\newcommand{\bits}{\{-1,1\}}
\newcommand{\bn}{\bits^n}

\renewcommand{\eps}{\varepsilon}
\renewcommand{\epsilon}{\varepsilon}
\usepackage{mathrsfs}

\def\colorful{1}

\ifnum\colorful=1

\newcommand{\red}[1]{{\color{red} {#1}}}

\fi
\ifnum\colorful=0

\newcommand{\red}[1]{{{#1}}}

\fi

\title{Quantitative Correlation Inequalities via Semigroup Interpolation}

\author{\hspace{-30pt}Anindya De \vspace{3pt}\\
\hspace{-38pt} \small{\sl University of Pennsylvania} \and \hspace{5pt}Shivam
Nadimpalli \vspace{3pt} \\ \small{\sl Columbia University} \and Rocco A.
Servedio \vspace{3pt}  \\ \small{\sl Columbia University}\vspace*{15pt}
}

\date{\small{\today}}

\begin{document}

\maketitle


\begin{abstract}

Most correlation inequalities for high-dimensional functions in the
literature, such as the Fortuin-Kasteleyn-Ginibre inequality and the celebrated Gaussian Correlation Inequality of Royen, are
\emph{qualitative} statements which establish that any two functions of a certain type
have non-negative correlation.  
In this work we give a general approach that can be used to bootstrap many
\emph{qualitative} correlation inequalities for functions over product spaces into
\emph{quantitative} statements.  The approach combines a new
extremal result about power series, proved using complex analysis, with 
harmonic analysis of functions over product spaces.
We instantiate this general approach in several different concrete settings 
to obtain a range of new and near-optimal quantitative correlation
inequalities, including:

\begin{itemize}

\item A {quantitative} version of Royen's celebrated Gaussian Correlation
Inequality \cite{roy14}. In \cite{roy14} Royen confirmed a conjecture, open for
40 years,
stating that any two symmetric convex sets must be non-negatively correlated
under any centered
Gaussian distribution.  We give a lower bound on the correlation in terms of
the vector of degree-2 Hermite coefficients of the two convex sets,
conceptually similar to Talagrand's
quantitative correlation bound for monotone Boolean functions over $\zo^n$ \cite{Talagrand:96}. We
show that our quantitative version of Royen's theorem is within a logarithmic
factor of being optimal.

\item A quantitative version of the well-known FKG inequality for monotone
functions over any finite product
probability space.  This is a broad generalization of Talagrand's quantitative correlation bound for functions from $\zo^n$ to $\zo$ under the uniform distribution
\cite{Talagrand:96}; the only prior generalization of which we are aware is due
to Keller \cite{Keller12,kell08-prod-measure,Keller09}, which extended
\cite{Talagrand:96} to product distributions over $\zo^n$. In the special case
of $p$-biased distributions over $\zo^n$ that was considered by Keller, our new
bound essentially saves a factor of $p \log(1/p)$ over the quantitative bounds
given in \cite{Keller12,kell08-prod-measure,Keller09}.  We also give two different quantitative
versions of the FKG inequality for monotone functions over the continuous domain
$[0,1]^n$, answering a question of Keller \cite{Keller09}.

\end{itemize}

\end{abstract}

\thispagestyle{empty}
\newpage 

\setcounter{page}{1}

\newpage


\section{Introduction} \label{sec:intro}

\emph{Correlation inequalities} are theorems stating that for certain classes of functions and certain probability distributions ${\cal D}$, any two functions $f,g$ in the class must be non-negatively correlated with each other under ${\cal D}$, i.e.~it must be the case that $\E_{\cal D}[f g] - \E_{\cal D}[f] \E_{\cal D}[g] \geq 0.$
 Inequalities of this type have a long history, going back at least to a well-known result of Chebyshev, ``Chebyshev's order inequality,'' which states that for any two nondecreasing sequences $a_1 \leq \cdots \leq a_n$, $b_1 \leq \cdots \leq b_n$ and any probability distribution $p$ over $[n]=\{1,\dots,n\}$, it holds that
 \[
 \sum_{i=1}^n a_i b_i p_i \geq \left(\sum_{i=1}^n a_i p_i \right) \left(\sum_{i=1}^n b_i p_i \right).
 \]
Modern correlation inequalities typically deal with high-dimensional rather than one-dimensional functions. 
Results of this sort have proved to be of fundamental interest in many fields such as combinatorics, analysis of Boolean functions, statistical physics, and beyond. 

Perhaps the simplest high-dimensional correlation inequality is the well known Harris-Kleitman theorem  \cite{harris60,kleitman66}, which states that if $f,g: \zo^n \to \zo$ are monotone functions (meaning that $f(x) \leq f(y)$ whenever $x_i \leq y_i$ for all $i$) then $\E[fg] - \E[f]\E[g] \geq 0$, where expectations are with respect to the uniform distribution over $\zo^n$.  
The Harris-Kleitman theorem has a one-paragraph proof by induction on $n$; on the other end of the spectrum is the Gaussian Correlation Inequality (GCI), which states that if $K,L \subseteq \R^n$ are any two symmetric convex sets and ${\cal D}$ is any centered Gaussian distribution over $\R^n$, then $\E_{\cal D}[K L] - \E_{\cal D}[K]\E_{\cal D}[L] \geq 0$ (where we identify sets with their 0/1-valued indicator functions).  
This was  a famous conjecture for four decades before it was proved by Thomas Royen in 2014 \cite{roy14}.
Other well-known correlation inequalities include the Fortuin-Kasteleyn-Ginibre (FKG) inequality \cite{fortuin1971}, which is an important tool in statistical mechanics and probabilistic combinatorics; the Griffiths--Kelly--Sherman (GKS) inequality \cite{griffiths1967correlations, kelly1968general}, which is a correlation inequality for ferromagnetic spin systems; and various generalizations of the GKS inequality to quantum spin systems \cite{gallavotti1971proof, suzuki1973correlation}.

\subsection{Quantitative Correlation Inequalities} \label{sec:qci}

The agenda of the current work is to obtain \emph{quantitative} correlation inequalities. Consider the following representative example: For two monotone Boolean functions $f, g: \zo^n \rightarrow \zo$, as discussed above, the Harris-Kleitman theorem states that $\Ex[f g] - \Ex[f] \Ex[g] \geq 0.$  
It is easy to check that the Harris-Kleitman inequality is tight if and only if $f$ and $g$ depend on disjoint sets of variables. One might therefore hope to get an improved bound by measuring how much $f$ and $g$ depend simultaneously on the same coordinates. Such a bound was obtained by Talagrand \cite{Talagrand:96} in an influential paper (appropriately titled ``How much are increasing sets correlated?").\ignore{
} To explain Talagrand's main result, we recall the standard notion of \emph{influence} from Boolean function analysis~\cite{od2014analysis}. For a Boolean function $f: \zo^n \rightarrow \zo$, the \emph{influence} of coordinate $i$ on $f$ is defined to be $\Inf_i[f] \coloneqq \Pr_{\bx \sim U_n} [f(\bx) \not = f(\bx^{\oplus i})]$, where $U_n$ is the uniform distribution on $\zo^n$ and $\bx^{\oplus i}$ is obtained by flipping the $i^\text{th}$ bit of $\bx$. Talagrand proved the following \emph{quantitative} version of the Harris--Kleitman inequality:
\begin{equation}~\label{eq:Talagrand}
\Ex[f g] - \Ex[f] \Ex[g] \ge \frac{1}{C} \cdot \Psi\left(\sum_{i=1}^n\Inf_i[f] \Inf_i[g]\right)
\end{equation} 
where $\Psi(x) := x/\log(e/x)$, $C>0$ is an absolute constant, and the expectations are with respect to the uniform measure. A simple corollary of this result is that 
$\Ex[f g] = \Ex[f] \Ex[g] $ if and only if the sets of influential variables for $f$ and $g$ are disjoint. In \cite{Talagrand:96} itself, Talagrand gives an example for which \Cref{eq:Talagrand} is tight up to constant factors.

Talagrand's result has proven to be  influential in  the theory of Boolean functions, and several works~\cite{kell08-prod-measure, Keller09, Keller12, kkm-corr} have obtained extensions and variants of this inequality for product distributions over $\zo^n$. An analogue of Talagrand's inequality in the setting of monotone functions over Gaussian space was obtained by Keller, Mossel and Sen~\cite{kms2-14} using a new notion of ``geometric influences.'' Beyond these results, we are not aware of \emph{quantitative} correlation inequalities in other settings, even though (as discussed above) a wide range of \emph{qualitative} correlation inequalities are known.  In particular, even for very simple and concrete settings such as the solid cube $[0,1]^n$ endowed with the uniform measure or the $m$-ary cube $\{0, 1, \ldots, m-1\}^n$ with a product measure, no quantitative versions of the FKG inequality were known (see the discussion immediately following Theorem~4 of Keller~\cite{keller2009lower}). As a final example, no quantitative version of the Gaussian Correlation Inequality was previously known.

\subsection{Our Contributions}

{\renewcommand{\arraystretch}{1.5}
\begin{table}
\centering
\begin{tabular}{@{}lcc@{}}
\toprule
& Qualitative Bounds         & Quantitative Bounds  \\ \midrule
Monotone $f, g \in L^2\pbra{\R^n, \gamma}$   & \cite{fortuin1971}
& \cite{kms2-14}                        \\
Symmetric, convex  $K, L \subseteq \R^n_\gamma$   & \cite{roy14}
& \Cref{thm:robust-gci}                        \\ 
Convex  $f, g \in L^2\pbra{\R^n, \gamma}$   & \cite{hu1997ito}
& \Cref{thm:our-hu}                        \\ 
Monotone $f, g: \zo^n \to \zo$                & \cite{harris60, kleitman66,
fortuin1971} & \cite{Talagrand:96}   \\ 
Monotone $f, g: \zo^n_{p} \to \R$  & \cite{fortuin1971}
& \cite{kell08-prod-measure}, \Cref{thm:our-talagrand} \\
Monotone $f, g: \{0, \ldots, m-1\}^n_{\pi} \to \R$ & \cite{fortuin1971}
& \Cref{thm:our-talagrand}                        \\
Monotone $f, g  \in L^2([0, 1]^n)$   & \cite{preston1974generalization} & \Cref{thm:our-talagrand-cont}, \Cref{thm:solid-talagrand-brownian}                       \\ \bottomrule
\end{tabular}
\caption{Qualitative and quantitative correlation inequalities. Here $\gamma$ denotes the standard Gaussian distribution ${\cal N}(0,1)^n$; $\zo^n_p$ denotes the $p$-biased hypercube (with no subscript corresponding to $p=1/2$, i.e.~the uniform distribution); $\pi$ denotes any  distribution over $\{0,\dots,m-1\}$; and $[0,1]^n$ is endowed with the Lebesgue measure.}
\end{table}
}

We establish a general framework to transfer qualitative correlation inequalities into quantitative correlation inequalities. We apply this general framework to obtain a range of new quantitative correlation inequalities, which include the following:
\begin{enumerate}
\item Quantitative versions of Royen's Gaussian Correlation Inequality and Hu's correlation inequality \cite{hu1997ito} for symmetric convex functions over Gaussian space;
\item A quantitative FKG inequality for a broad class of  product distributions, including arbitrary product distributions over finite domains and the uniform distribution over $[0,1]^n$.
\end{enumerate}
All these results are obtained in a unified fashion via simple proofs that are substantially different from previous works~\cite{Talagrand:96, kell08-prod-measure, Keller09, Keller12, kkm-corr}.  We also give several lower bound examples, including one which shows that our quantitative version of the Gaussian Correlation Inequality is within a logarithmic factor of the best possible bound. 

We note that the special case of item~2 above with the uniform distribution on $\zo^n$ essentially recovers Talagrand's correlation inequality \cite{Talagrand:96}. In more detail, our bound is weaker than that obtained in \cite{Talagrand:96} by a logarithmic factor, but our proof is significantly simpler and easily generalizes to other domains. For $p$-biased distributions over $\zo^n$, our bound avoids any dependence on $p$ compared to  the results of Keller~\cite{kell08-prod-measure, Keller09, Keller12}  which have a $p \log(1/p)$ dependence (though, similar to the situation vis-a-vis \cite{Talagrand:96}, we lose a logarithmic factor in other dependencies). 

\subsection{The Approach}\label{sec:ourapproach}

We start with a high level meta-observation before explaining our framework and techniques in detail. While the statements  of the Harris-Kleitman inequality, the FKG inequality, and the Gaussian Correlation Inequality have a common flavor, the proofs of these results are extremely different from each other. (As noted earlier, the Harris-Kleitman inequality admits a simple inductive proof which is only a few lines long; in contrast the Gaussian Correlation Inequality was an open problem for nearly four decades, and no inductive proof for it is known.) Thus, at first glance, it is not clear how one might come up with a common framework to obtain quantitative versions of these varied qualitative inequalities. 

Our approach circumvents this difficulty by using the qualitative inequalities essentially as ``black boxes.'' This allows us to extend the qualitative inequalities into quantitative ones while essentially sidestepping the difficulties of proving the initial qualitative statements themselves. 

\subsubsection{Our General Framework} \label{subsubsec:general-framework}

In this subsection, we give an overview of our general framework and the high-level ideas underlying it, with our quantitative version of the Gaussian Correlation Inequality serving as a running example throughout for concreteness. 

We begin by defining a function $\Phi: [0,1] \rightarrow [0,1]$
 which will play an important role in our results:  
 \begin{equation} \label{eq:Phi}
 \Phi(x) := \min \left\{x,  {\frac x {\log^2(1/x)}} \right\}.
 \end{equation}  
 (Note the similarity between $\Phi$ and the function $\Psi$  mentioned earlier that arose in Talagrand's quantitative correlation inequality \cite{Talagrand:96}; the difference is that $\Phi$ is smaller by essentially a logarithmic factor in the small-$x$ regime.)

 Let $\mathcal{F}$ be a family of real-valued functions on some domain (endowed with measure $\mu$) with $\E_{\mu}\left[f^2\right] \leq 1$ for all $f\in\calF$. For example, the Gaussian Correlation Inequality is a correlation inequality for the family $\Fcsc$ of centrally symmetric, convex sets (identified with their $0/1$-indicator functions), and $\mu$ is the standard Gaussian measure $\calN(0,1)^n$, usually denoted $\gamma$.\footnote{Since convexity is preserved under linear transformation, no loss of generality is incurred in assuming that the background measure is the standard normal distribution ${\cal N}(0,1)^n$ rather than an arbitrary centered Gaussian.} A \emph{quantitative} correlation inequality for $f, g \in {\cal F}$ gives a (non-negative) lower bound on the quantity $\Ex_{\bx \sim \mu}[f(\bx) g(\bx)] - \Ex_{\bx \sim \mu}[f(\bx)] \Ex_{\by \sim \mu} [g(\by)]$. typically in terms of some measure of ``how much $f$ and $g$ simultaneously depend on the same coordinates.'' Our general approach establishes such a quantitative inequality in two main steps:

 \medskip
\noindent {\bf Step 1:}  For this step, we require an appropriate
family of ``noise operators" $(\T_\rho)_{\rho \in [0,1]}$ with respect to the measure $\mu$. Very briefly, each of these operators $\T_\rho$ will be a (re-indexed version of a) symmetric Markov operator whose stationary distribution is $\mu$; this is defined more precisely in \Cref{sec:general-approach}.  (Looking ahead, we will see, for example, that in the case of the GCI, the appropriate noise  operator is the Ornstein-Uhlenbeck noise operator, defined in \Cref{def:ou-operator}.) The crucial property we require of the family $(\T_\rho)_{\rho \in [0,1]}$ with respect to $\mathcal{F}$ is what we refer to as \emph{monotone compatibility}:

\begin{definition}[Monotone compatibility] \label{def:monotone-compatible1}
A class of functions $\mathcal{F}$ and background measure $\mu$ is said to be \emph{monotone compatible} with respect to a family of noise operators $(\T_\rho)_{\rho \in [0,1]}$ if (i) for all $f, g \in \mathcal{F}$, the function 
\[
q(\rho) := \Ex_{\bx \sim \mu} [f(\bx) \T_{\rho} g(\bx)]
\]
is a non-decreasing function of $\rho$, and (ii) for $\rho=1$ we have $\T_1=\mathrm{Id}$ (the identity operator).
\end{definition}

The notion of monotone compatibility should be seen as a mild extension of qualitative correlation inequalities. As an example,\ignore{Step~1 in our proof sketch of \Cref{eq:Talagrand-our} shows that the family $\Fmon$ of monotone Boolean functions is monotone compatible with the family of Bonami-Beckner operators. Similarly,} in the case of the Gaussian Correlation Inequality, Royen's proof~\cite{roy14} in fact shows that that the family $\Fcsc$ is monotone compatible with Ornstein-Uhlenbeck operators. 
 
\medskip
 
\noindent \textbf{Step 2:}  We express the operator $\T_{\rho}$ in terms of its eigenfunctions. 
In all the cases we consider in this paper, the eigenvalues of the operator $\T_{\rho}$ are $\left\{\rho^j\right\}_{j \ge 0}$. Let $\{\mathcal{W}_j\}_{j \ge 0}$ be the corresponding eigenspaces. Consequently, we can express $q(\rho)-q(0)$ as
\begin{equation} \label{eq:recall-me}
q(\rho)-q(0)=
\Ex_{\bx \sim \mu}[f(\bx) \T_{\rho} g(\bx)] - \Ex_{\bx \sim \mu}[f(\bx)] \cdot \Ex_{\by \sim \mu}[g(\by)] = \sum_{j >0} \rho^j \Ex[f_j (\bx) g_j (\bx)],
\end{equation}
where $f_j$ (respectively $g_j$) is the projection of $f$ (respectively $g$) on the space $\mathcal{W}_j$. To go back to our running example, for the Gaussian Correlation Inequality, $\mathcal{W}_j$ is the subspace spanned by degree-$j$ Hermite polynomials on $\mathbb{R}^n$. 

Define $a_j \coloneqq \Ex[f_j (\bx) g_j (\bx)]$, so $q(\rho) = \sum_{j \geq 0} a_j \rho^j.$ Now, corresponding to any family $\mathcal{F}$ and noise operators $(\T_{\rho})_{\rho \in [0,1]}$, there will be a unique $j^\ast \in \mathbb{N}$ such that the following properties hold:
\begin{enumerate}
	\item If $a_{j^\ast} =0$, then $\Ex_{\bx \sim \mu}[f(\bx)  g(\bx)] = \Ex_{\bx \sim \mu}[f(\bx)] \cdot \Ex_{\by \sim \mu}[g(\by)]$. In other words, $a_{j^\ast}$ \emph{qualitatively} captures the ``slack" in the correlation inequality (in fact, as we will soon see, $a_{j^\ast}$ also gives a quantitative lower bound on this slack). For example, for the Gaussian Correlation Inequality, it turns out that $j^\ast=2$ (for most of the other applications of our general framework in this paper, it turns out that $j^\ast = 1$).
	\item For any $i$ such that $j^\ast$ does not divide $i$, $a_i=0$. 
\end{enumerate}

Now, from the fact that the spaces $\{\mathcal{W}_j\}$ are orthonormal and the fact that every $f \in {\cal F}$ has $\E_{\mu}\left[f^2\right] \leq 1$, it follows that $\sum_{j>0} |a_j| \le 1$. 
Our main technical lemma, \Cref{lem:main-lem}, implies (see the proof of \Cref{thm:main-thm}) that for any such power series $q(\cdot),$ there exists some $\rho^\ast \in [0,1]$ such that 
\begin{equation}\nonumber
q(\rho^\ast) - q(0) \ge \frac{1}{C} \cdot \Phi (a_{j^\ast}). 
\end{equation}
  The proof crucially uses tools from complex analysis.
As the class ${\cal F}$ is monotone compatible with the operators $(\T_{\rho})_{\rho \in [0,1]}$, recalling \Cref{eq:recall-me}, it follows that  
\begin{equation}~\label{eq:meta-Talagrand}
q(1) - q(0) = \Ex_{\bx \sim \mu}[f(\bx)  g(\bx)] - \Ex_{\bx \sim \mu}[f(\bx)] \cdot \Ex_{\by \sim \mu}[g(\by)]  \ge \frac{1}{C} \cdot \Phi (a_{j^\ast}),
\end{equation}
which is the desired quantitative correlation inequality for ${\cal F}$.

\begin{remark}
We emphasize the generality of our framework; the argument sketched above can be carried out in a range of different concrete settings.  For example,  by using the Harris-Kleitman qualitative correlation inequality for monotone Boolean functions in place of the GCI, and the Bonami-Beckner noise operator over $\zo^n$ in place of the Ornstein-Uhlenbeck noise operator, the above arguments give a simple proof of the following (slightly weaker) version of Talagrand's correlation inequality (\Cref{eq:Talagrand}):
\begin{equation}~\label{eq:Talagrand-our}
\Ex[f g] - \Ex[f] \Ex[g] \ge \frac{1}{C} \cdot \Phi\left( \sum_{i=1}^n \Inf_i[f] \Inf_i[g]\right),  
\end{equation}
for an absolute constant $C>0$. 
While our bound is weaker than that of \cite{Talagrand:96} by a log factor (recall the difference between $\Psi$ and $\Phi$), our methods are applicable to a much wider range of settings (such as the GCI and the other applications given in this paper).
Finally, we emphasize that our proof strategy is really quite different from that of \cite{Talagrand:96}; for example, \cite{Talagrand:96}'s proof relies crucially on bounding the degree-2 Fourier weight of monotone Boolean functions by the degree-1 Fourier weight, whereas our strategy does not analyze the degree-2 spectrum of monotone Boolean functions at all.
\end{remark}

\begin{remark}
Coupled with the first property described above, \Cref{eq:meta-Talagrand} shows that $a_{j^\ast}$  not only qualitatively captures the ``correlation gap'' 
\[
\Ex_{\bx \sim \mu}[f(\bx)  g(\bx)] - \Ex_{\bx \sim \mu}[f(\bx)] \cdot \Ex_{\by \sim \mu}[g(\by)],
\]
but also provides a quantitative lower bound on this gap. The following subsection discusses the semantics of the quantity $a_{j^\ast}$, {as well as the concrete instantiations of our general framework that we consider,} in more detail.
\end{remark}

\subsection{The Notion of Influence and Interpretation of $a_{j^\ast}$}
In both \Cref{eq:Talagrand} and in \Cref{eq:Talagrand-our}, the advantage obtained is expressed in terms of the inner product of influence vectors of the two functions $f$ and $g$.\footnote{Recall that for monotone Boolean functions, $\wh{f}(i) = \Inf_i[f]$.} Here we explain how the advantage terms obtained in several other settings considered in this paper also admit similar interpretations. 

\subsubsection{The Setting of Centrally Symmetric, Convex Sets over Gaussian Space}

Let $f,g : \mathbb{R}^n \rightarrow \{0,1\}$ be the indicator functions of centrally symmetric, convex sets. Let $\mathcal{W}
_j$ be the space of degree-$j$ Hermite polynomials  and let $f_j$ (respectively $g_j$) be the projection of $f$ (respectively $g$) on $\mathcal{W}_j$.  
Define $a_j \coloneqq \Ex_{\bx \sim \gamma_n} [ f_j (\bx) \cdot g_j(\bx)].$ We show that the following properties hold:
\begin{enumerate}
\item For any centrally symmetric, convex sets $f$ and $g$, $a_2 \ge 0$. Furthermore, $a_2 = 0$ if and only if there are orthogonal subspaces $U$, $V$ such that $\mathbb{R}^n = U \oplus V$, $f(x)$ depends only on $x_U$ (i.e. the projection of $x$ on $U$) and $g(x)$ depends only on $x_V$ (except perhaps on a set of measure zero). It follows that in this case, 
$\Ex_{} [f(\bx) g(\bx)] = \Ex[f(\bx)] \Ex[ g(\bx)]$.
\item For symmetric convex sets $f$ and $g$ and any odd $j$, $a_j$ is zero. 
\end{enumerate}
Thus, the two properties required of $a_{j^\ast}$ (mentioned in~\Cref{subsubsec:general-framework}) are both satisfied in this case with $j^\ast=2$. The second property above is an immediate consequence of $f$ and $g$ being centrally symmetric (and hence even\footnote{Recall that a function $f: \R^n \to \R$ is said to be \emph{even} if $f(x) = f(-x)$.}). The first property above crucially relies on the following new notion. For a unit vector $v \in \mathbb{S}^{n-1}$, define $\Inf_v[f]$ as
\[
\Inf_v[f] \coloneqq \Ex_{\bx \sim \gamma_n} \bigg[ f(\bx) \cdot \bigg(\frac{\langle \bx, v\rangle^2-1}{\sqrt{2}}\bigg)\bigg] = \Ex_{\bx \sim \gamma_n} \bigg[ f_2(\bx) \cdot \bigg(\frac{\langle \bx, v \rangle^2-1}{\sqrt{2}}\bigg)\bigg]. 
\]

As we show in \Cref{lem:influence-basics}, when $f$ is the indicator of a centrally symmetric, convex set, the quantity $\Inf_v[f]$ has two crucial properties:
\begin{enumerate}
\item For all $v \in \mathbb{S}^{n-1}$, $\Inf_v[f] \ge 0$. 
\item If $\Inf_v[f] =0$, then $f(x)$ only depends on $x_{{v}^{\perp}}$ (the projection of $x$ on $v^{\perp}$) (except perhaps on a set of measure zero).  In other words, $f(x) = f(y)$ whenever $x_{{v}^{\perp}} = y_{v^{\perp}}$.
\end{enumerate}
Thus $\Inf_v[f]$ qualitatively behaves like ``the influence of $f$ along direction $v$;'' indeed, this is why we chose the notation $\Inf_v[f].$ With this definition, it can be shown that there is a choice of an orthonormal basis $\{v_1, \ldots, v_n\}$ such that 
\begin{equation}~\label{eq:influences-rotation-basis}
a_2 = \sum_{i=1}^n \Inf_{v_i}[f] \Inf_{v_i}[g].  
\end{equation}
Thus the quantity $a_2$ for the quantitative Gaussian Correlation Inequality resembles the quantity $a_1$ for Talagrand's correlation inequality for monotone Boolean functions; recall that in the latter setting, $a_1 = \sum_{i=1}^n \Inf_i[f] \Inf_i[g]$. It can in fact be shown that the quantity $\Inf_v[f]$ satisfies several other familiar properties that are satisfied by the standard notion of influence over the discrete cube~\cite{od2014analysis}; these will be elaborated on in a future paper~\cite{Influence-future}.

\subsubsection{The $m$-ary Cube $\{0,\ldots, m-1\}^n$}

We consider real-valued functions $f: \{0,1,\ldots,m-1\}^n \rightarrow \mathbb{R}$ where the domain is endowed with a product measure $\pi^{\otimes n}$ with $\pi$ an arbitrary distribution over $\{0,1,\dots,m-1\}$ with full support. Recall that a real-valued function over the domain $\{0,1,\dots,m-1\}^n$ is said to be monotone if for all $x, y \in \{0,1,\dots,m-1\}^n$, if $x_i \geq y_i$ for all $i \in [n]$, then $f(x) \geq f(y)$. For $i\in[n]$, we  define 
\[
f^{=\{i\}} (x) = \Ex_{\by \sim \pi^{\otimes n}}[f(\by_1, \ldots, \by_{i-1}, x_i, \by_{i+1}, \ldots, \by_n)] - \Ex_{\by \sim \pi^{\otimes n}}[f(\by_1,\ldots, \by_n)]. 
\]
The function $f^{=\{i\}}(x)$ can be thought of as the  projection of $f$ onto the space of all functions which only depend on $x_i$ and are orthogonal to the constant function. 
In fact, the functions $\{f^{=\{i\}}\}_{i=1}^n$ correspond  to the first level of the well-studied \emph{Efron-Stein decomposition} of $f$ (see~\Cref{thm:es} for more details). 

Let $\Fmon$ be the family of monotone functions over $\{0,1,\dots,m-1\}^n$. Then, for $f, g \in \Fmon$, we have 
\[
\Ex_{\bx \sim \pi^{\otimes n}}[f(\bx) g(\bx) ] - \Ex_{\bx \sim \pi^{\otimes n}}[f(\bx)] \Ex_{\by \sim \pi^{\otimes n}}[g(\by)] \ge \frac{1}{C} \cdot \Phi(a_{1}), 
\]
where $a_1 = \sum_{i=1}^n \Ex_{\bx \sim \pi^{\otimes n}}[f^{=\{i\}}(\bx) g^{=\{i\}}(\bx)]$. 
For monotone $f$ and $g$, the quantity $a_1$ again satisfies two important properties (see~\Cref{lem:monotone-uncorrelated}): 
\begin{enumerate}
\item $a_1 \ge 0$. 
\item If $a_1 =0$, then for every $1 \le i \le n$, 
either $f^{=\{i\}}$ or $g^{=\{i\}}$ is identically zero. Consequently, there exists a partition of $[n] = S \sqcup \overline{S}$ such that $f$ (respectively $g$) only depends on the variables in $S$ (respectively $\overline{S}$). 
\end{enumerate}
Thus, for $j^\ast=1$, the quantity $a_{j^\ast}$ again satisfies the two properties stated in \Cref{subsubsec:general-framework}, and is analogous to the quantity $\sum_{i=1}^n \Inf_i[f] \Inf_i[g]$ in Talagrand's correlation inequality (\Cref{eq:Talagrand}) for the uniform distribution over $\zo^n$. Therefore, the quantity $a_{j^\ast}$ intuitively captures ``how much $f$ and $g$ depend simultaneously on the same coordinates.''

\subsubsection{The Solid Cube $[-1,1]^n$}

Finally, we give two different quantitative correlation inequalities for monotone functions over the solid cube.  For the first one, we observe that the above discussion for monotone functions over $\{0,1,\dots,m-1\}^n$ goes through essentially unchanged for monotone functions over  $[-1,1]^n$ (where for simplicity we take the background measure to be the uniform measure over $[-1,1]^n$).  In this setting, entirely analogous to the finite product measure situation discussed above, we have $j^\ast=1$ and $a_1 = \sum_{i=1}^n \E[f^{=\{i\}}(\bx) g^{=\{i\}}(\bx)]$, where $f^{=\{i\}}$ is the $\{i\}$-component of the Efron-Stein decomposition of $f$ with respect to the uniform measure over $[-1,1]^n$. 

Our second (incomparable) correlation inequality over the solid cube $[0,1]^n$ is obtained by using a different noise operator, namely, the Markov semigroup associated with \emph{reflected} Brownian motion (i.e. Brownian motion on $[0,1]^n$ with Neumann boundary conditions, see \cite{bakry2013analysis}). The lower bound on the correlation is then obtained in terms of the degree-1 coefficients of the \emph{cosine basis} for functions on $[0,1]^n$.\ignore{---see \Cref{subsec:talagrand-brownian}.}

These correlation inequalities thus give an answer (or rather, two answers) to a question posed by Keller \cite{Keller09}, who wrote ``It seems tempting to find a generalization of Talagrand's result to the continuous setting, but it is not clear what is the correct notion of influences in the continuous case that should
be used in such generalization.''

\subsection{Organization}

The rest of the paper is organized as follows: \Cref{sec:prelims} recalls the necessary background on Markov semigroups and functional analysis, and recalls a well-known result from complex analysis that we will require to prove our main lemma. In \Cref{sec:main-lemma}, we prove our main technical lemma, \Cref{lem:main-lem}, which is at the heart of our approach to quantitative correlation inequalities. \Cref{sec:general-approach} presents our general approach to quantitative correlation inequalities, \Cref{thm:main-thm}, which we instantiate with concrete examples in subsequent sections as follows:
In \Cref{sec:application-royen}, we obtain quantitative analogues of several correlation inequalities over Gaussian space; in particular, we give robust forms of Royen's Gaussian Correlation Inequality (GCI) \cite{roy14}, present an extension of the quantitative GCI to quasiconcave functions, and also obtain a robust form of Hu's correlation inequality for convex functions \cite{hu1997ito}.
In \Cref{sec:application-talagrand}, we obtain an analogue of Talagrand's correlation inequality \cite{Talagrand:96} in the setting of monotone functions over finite product spaces. We note that this setting includes the Boolean hypercube $\zo^n$, wherein we obtain a generalization of Talagrand's inequality to real-valued functions and $p$-biased distributions.
Finally, in \Cref{sec:talagrand-solid-cube}, we give our two quantitative correlation inequalities for monotone functions over the solid cube $[-1,1]^n$.


\section{Preliminaries} \label{sec:prelims}

In this section we give preliminaries setting notation, recalling useful background on noise operators and orthogonal decomposition of functions over product spaces, and recalling a well-known result that we will require from complex analysis.

\subsection{Noise Operators and Orthogonal Decompositions}
Let $(\Omega, \pi)$ be a probability space; we do not require $\Omega$ to be finite, and we assume without loss of generality that $\pi$ has full support. 

The background we require for noise operators on functions in $L^2(\Omega, \pi)$ is most naturally given using the language of ``Markov semigroups.''  Our exposition below will be self-contained; for a general and extensive resource on Markov semigroups, we refer the interested reader to \cite{bakry2013analysis}.

\begin{definition}[Markov semigroup] \label{def:markov-semigroup}
	A collection of linear operators $\left(\operatorP_t\right)_{t\geq0}$ on $L^2(\Omega, \pi)$ is said to be a \emph{Markov semigroup} if 
	
	\begin{enumerate}
	
	\item 
	$\operatorP_0 = \mathrm{Id}$;
	\item for all $s, t \in [0, \infty)$, we have $\operatorP_s\circ \operatorP_t = \operatorP_{s + t}$; and
	\item for all $t \in [0, \infty)$ and all $f, g \in L^2(\Omega, \pi)$, the following hold:
	\begin{enumerate}
		\item \emph{Identity}: $\operatorP_t1 = 1$ where $1$ is the identically-$1$ function.
		\item \emph{Positivity}: $\operatorP_t f \geq 0$ almost everywhere if $f \geq 0$ almost everywhere.\footnote{Note that this implies the following \emph{order} property: if $f \geq g$ almost everywhere, then $\operatorP_t f \geq \operatorP_t g$ almost everywhere.}
	\end{enumerate}

	\end{enumerate}
\end{definition}

It is well known that a Markov semigroup can be constructed from a Markov process and vice versa \cite{bakry2013analysis}. We call a Markov semigroup \emph{symmetric} if the underlying Markov process is time-reversible; the following definition is an alternative elementary characterization of symmetric Markov semigroups.  (Recall that for $f,g \in L^2(\Omega,\pi)$ the inner product $\langle f, g \rangle$ is defined as $\E_{\bx \sim \pi}[f(\bx) g (\bx)]$.)

\begin{definition}[Symmetric Markov semigroup] \label{def:symmetric-markov-semigroup}
A Markov semigroup $\left(\operatorP_t\right)_{t\geq0}$ on $L^2(\Omega, \pi)$ is \emph{symmetric} if for all $t\in[0,\infty)$, the operator $\operatorP_t$ is self-adjoint; equivalently,  for all $t\in[0,\infty)$ and all $f, g\in L^2(\Omega, \pi)$, we have $\la f, \operatorP_t g\ra = \la \operatorP_t f, g\ra$. 
\end{definition}

We note that the families of noise operators $(\U_\rho)_{\rho \in [0,1]}$ and $(\T_{\rho})_{\rho \in [0,1]}$ that we consider in  \Cref{sec:application-royen} and \Cref{sec:application-talagrand} respectively will be parametrized by $\rho \in [0,1]$ where $\rho = e^{-t}$ for $t\in[0,\infty)$, as is standard in theoretical computer science. (For example, the Bonami-Beckner noise operator operator $\T_\rho$ mentioned in the Introduction, which is a special case of the $\T_\rho$ operator defined in \Cref{sec:application-talagrand}, corresponds to $\operatorP_t$ for $(\operatorP_t)_{t \geq 0}$ a suitable Markov semigroup and $\rho = e^{-t}.$)
 
An important operator associated to every Markov semigroup is the so-called \emph{infinitesimal generator} or simply \emph{generator} of the semigroup.

\begin{definition} \label{def:generator}
	The \emph{infinitesimal generator} or \emph{generator} of a Markov semigroup $\left(\operatorP_t\right)_{t\geq0}$ on $L^2(\Omega, \pi)$ is the operator $\calL$ given by
	$$\calL f = -\lim_{t\to0} \frac{\operatorP_t f - f}{t}.$$
	We will denote the domain of $\calL$ by $\dom{\calL}$. 
\end{definition}

\begin{remark}
	
	It may be the case that $\calL$ is sometimes only defined on a proper subset of $L^2(\Omega, \pi)$, i.e. $\dom{\calL} \subsetneq L^2(\Omega, \pi)$; for example, when $\Omega$ is not discrete, it is only possible to define $\calL$ on a dense subset of $\Omega$. There is extensive technical literature on this subject (see, for example, \cite{bakry2013analysis}).
\end{remark}

Just as a Markov semigroup $\left(\operatorP_t\right)_{t\geq0}$ determines its generator $\calL$, a generator $\calL$ and its domain $\dom{\calL}$ together also determine the Markov semigroup $\left(\operatorP_t\right)_{t\geq0}$. In particular, the operator $\P_t$ can be recovered as the solution of the \emph{Kolmogorov equation}:
$$\frac{d}{dt} \P_t f = -\P_t \calL f, \qquad \P_0f = f.$$
Note that this lets us formally write $\P_t = e^{-t\calL}$; this expression readily makes sense as a power series when $\dom{\calL} = L^2(\Omega, \pi)$, but when  $\dom{\calL} \subsetneq L^2(\Omega, \pi)$, the meaning of the exponential function must be carefully defined. 

Given a Markov semigroup $\left(\operatorP_t\right)_{t\geq0}$ on the probability space $(\Omega, \pi)$, we can naturally define the Markov semigroup $\left( \otimes_{i=1}^n \operatorP_{t_i}\right)_{t_i \geq 0}$ on $L^2\left(\Omega^n, \pi^{\otimes n}\right)$.  
We write $\operatorP_{\overline{t}}$ to denote this semigroup, and write $\operatorP_t$ to denote the Markov semigroup $(\otimes_{i=1}^n \operatorP_t)_{t \geq 0}$. We next define a decomposition of $L^2\left(\Omega^n, \pi^{\otimes n}\right)$ that is particularly well-suited to the action of a Markov semigroup $\left(\operatorP_t\right)_{t\geq0}$.

\begin{definition}[Chaos decomposition] \label{def:chaos-decomposition}
	Consider a Markov semigroup $\left(\operatorP_t\right)_{t\geq0}$ on $L^2\left(\Omega^n, \pi^{\otimes n}\right)$. We call an orthogonal decomposition of  
	$$L^2\left(\Omega^n, \pi^{\otimes n}\right) = \bigoplus_{i=0}^\infty \calW_i$$ 
a \emph{chaos decomposition} with respect to the Markov semigroup $\left(\operatorP_t\right)_{t\geq 0}$ if 
\begin{enumerate}
	\item $\calW_0 = \mathrm{span}(1)$ where $1$ is the identically-1 function (i.e.~$\calW_0=\R$). 
	\item For all $t\geq0$, there exists $\lambda_t \in [0,1]$ such that if $f \in \calW_i$, then $\operatorP_t f = \lambda_t^{i} f$.
	\item If $t_1 > t_2$, then $\lambda_{t_1} < \lambda_{t_2}$.
\end{enumerate}
\end{definition}

{The term ``chaos decomposition'' is used in the literature to describe the spectral decomposition of $L^2(\R^n, \gamma)$ with respect to the Laplacian of the Ornstein--Uhlenbeck semigroup (see \Cref{fact:gaussian-noise-expansion}); its usage in the broader sense defined above is not standard (to our knowledge).} 

\begin{notation}
	Given an orthogonal decomposition $L^2\left(\Omega^n, \pi^{\otimes n}\right) = \bigoplus_{i} \calW_i$,  for $f \in L^2\left(\Omega^n, \pi^{\otimes n}\right)$ we will write $f = \oplus_{i} f_i$ where $f_i$ is the projection of $f$ onto $\calW_i$.
\end{notation}

We note that $\lambda_0 = 1$, and as $1 \in \calW_0$, it follows that $f_0 = \la f, 1\ra$. {We revisit the definition of monotone compatibility given in the Introduction in the language of Markov semigroups:}

\begin{definition}[Monotone compatibility] \label{def:monotone-compatible}
	Let $\left(\operatorP_t\right)_{t\geq0}$ be a Markov semigroup on $L^2\left(\Omega^n, \pi^{\otimes n}\right)$. We say that $\left(\operatorP_t\right)_{t\geq0}$ is \emph{monotone compatible} with a family of functions $\calF \subseteq L^2\left(\Omega^n, \pi^{\otimes n}\right)$ if for all $f, g \in \calF$, we have 
	$$\frac{\partial}{\partial t} \la \operatorP_t f, g\ra \leq 0.$$
\end{definition}

Recalling that our noise operators such as $(\T_\rho)_{\rho \in [0,1]}$ are reparameterized versions of the Markov semigroup operators $(\operatorP_t)_{t \geq 0}$ under the reparameterization $\T_\rho = \operatorP_t$ with $\rho = e^{-t}$, and recalling item~1 in \Cref{def:markov-semigroup}, we see that \Cref{def:monotone-compatible} coincides with \Cref{def:monotone-compatible1}.

\begin{example} \label{ex:monotone-boolean}
To provide intuition for \Cref{def:monotone-compatible}, a useful concrete example to consider is 

\begin{itemize}

\item $\Omega=\zo$ and $\pi = $ the uniform distribution on $\Omega$, so $L^2\left(\Omega^n,\pi^{\otimes n} \right)$ is the space of all real-valued functions on the Boolean cube $\zo^n$ under the uniform distribution;

\item $\Fmon = $ the class of all monotone Boolean functions, i.e. all $f: \zo^n \to \zo$ such that if $x_i \leq y_i$ for all $i$ then $f(x) \leq f(y)$;

\item $(\operatorP_t)_{t \geq 0}$ is defined by $\operatorP_t = \T_{e^{-t}}$, where $\T_\rho$ is the Bonami-Beckner operator.  We remind the reader that for any $f: \zo^n \rightarrow \mathbb{R}$ and any $0 \leq \rho \leq 1$, the function $\T_\rho f (x)$ is defined to be 
$ \E_{\by \sim N_\rho(x)}[f(\by)],$
where ``$\by \sim N_\rho(x)$'' means that $\by \in \zo^n$ is randomly chosen by independently setting each $\by_i$ to be $x_i$ with probability $\rho$ and to be uniform random with probability $1-\rho$.

\end{itemize}

In this setting, as will be shown later, we have that for any two monotone functions $f,g \in \calF$, the
function $\Ex_{\bx \sim \bn}[\T_{\rho}f(\bx) g(\bx)]$ is a non-decreasing function of $\rho$; hence $\frac{\partial}{\partial t}\Ex_{\bx \sim \bn}[\operatorP_{t}f(\bx) g(\bx)]$ is always at most 0 (note that as $t$ increases $\rho = e^{-t}$ decreases), so $(\operatorP_t)_{t \geq 0}$ is monotone compatible with $\Fmon$.
\end{example}

\subsection{Complex Analysis} \label{subsec:complex-anal} 

Let $U \subseteq \C$ be a connected, open set. Recall that a function $f : U \to
\C$ is said to be \emph{holomorphic} if at every point in $U$ it is
complex differentiable in a neighborhood of the point. For $U$ a connected closed set, $f$ is said to be holomorphic if it is holomorphic in an open set containing $U$. Our main technical lemma appeals to the following classical result, a proof of which can be found in
\cite{rudin87}.

\begin{theorem}[Hadamard Three Circles Theorem]
\label{thm:three-circles}
	Suppose $f$ is holomorphic on the annulus $\{z\in \C \mid r_1 \leq |z| \leq
	r_2\}$. For $r \in [r_1, r_2]$, let $M(r) := \max_{|z| = r} |f(z)|.$ Then
	$$\log\left(\frac{r_2}{r_1}\right)\log M(r) \leq
	\log\left(\frac{r_2}{r}\right)\log M(r_1)
	+ \log\left(\frac{r}{r_1}\right)\log M(r_2).$$
\end{theorem}


\section{A New Extremal Bound for Power Series with Bounded Length} \label{sec:main-lemma}

Given a complex power series $p(t) = \sum_{i=1}^\infty c_it^i$ where $c_i \in \C$, its \emph{length} is defined to be the sum of the absolute values of its coefficients, i.e. $\sum_{i=1}^\infty |c_i|$. Our main technical lemma is a lower bound on the sup-norm of complex power series with no constant term and bounded length:\footnote{The ``3/2'' in the lemma below could be replaced by any constant bounded above 1; we use 3/2 because it is convenient in our later application of \Cref{lem:main-lem}.}

\medskip

\begin{lemma}[Main Technical Lemma]
	\label{lem:main-lem}
Let $p(t) = \sum_{i=1}^\infty c_it^i$ with $c_1 = 1$ and $\sum_{i=1}^\infty
|c_i| \leq M$ where $M \geq 3/2$. Then:
	$$\sup_{t\in [0, 1]} |p(t)| \geq \frac{\Theta(1)}{\log^2 M}.$$
\end{lemma}

The proof given below is inspired by arguments with a similar flavor  in \cite{BE97,BEK99}, where the Hadamard Three Circles Theorem is used to prove various extemal bounds on polynomials. 

\begin{proof}
Consider the 
meromorphic map (easily seen to have a single pole at $z=0$) given by
	$$h(z) = A\pbra{z + \frac{1}{z}} + B,$$
	which maps origin-centered circles to ellipses centered at $B$.
Let $0 < \delta < c$ be a parameter that we will fix later, where $0<c<1$ is an absolute constant that will be specified later. We impose the
following constraints on $A$ and $B$:
	$$-2A + B = \delta \qquad\qquad \frac{17}{4}A + B = 1,$$
and note that these constraints imply that $A = \frac{4(1-\delta)}{25}$ and $B = \frac{8 +
17\delta}{25}$.

We define three circles in the complex plane that we will use for the Hadamard Three Circles Theorem:
\begin{itemize}
	\item [(i)]Let $C_1$ be the circle centered at 0 with radius 1. Note that for all $z \in C_1$, the value $h(z)$ is a real number in the interval  $\sbra{\delta, \frac{16 + 9\delta}{25}} \subseteq [\delta, 1)$.
	\item [(ii)] Let $r>1$ be such that $h(-r) = 0$, so $r + {\frac 1 r} = {\frac {8+17\delta}{4-4\delta}} = 2 + \Theta(\delta)$ and hence $r = 1 + \Theta(\sqrt{\delta})$, which is less than 4. Define $C_2$ to be the circle
	centered at 0 with radius $r$. 
	\item [(iii)] Let $C_3$ be the circle centered at 0 with radius 4. Note that
	$|h(z)| \leq 1$ for $z \in C_3$.
\end{itemize}

Define $q(t) := \frac{p(t)}{t}$. 
Note that $q(0) = c_1 = 1$ and that for all $z \in \C$ such that $|z| \leq 1$, we have $|q(z)| \leq M$.  
Define $\psi(z) := q(h(z))$. 
Note that $\psi$ is holomorphic on $\mathbb{C}\backslash\{0\}$; in particular, it is holomorphic on the annulus defined by $C_1$ and $C_3$. 
Consequently, by~\Cref{thm:three-circles}, we have:
$$\log\pbra{\frac{4}{1}}\log \alpha(r) \leq \log\pbra{\frac{4}{r}}\log
\alpha(1) + \log\pbra{\frac{r}{1}}\log \alpha\pbra{4}$$
with $\alpha(r) := \sup_{|z| = r} |\psi(z)|$. 
 As $h(-r) = 0$, we have $\psi(-r)
= 1$ and so $\log \alpha(r) \geq 0$. Consequently, the left hand side of the
above inequality is non-negative, which implies:
	$$1 \leq \alpha(1)^{\log\pbra{\frac{4}{r}}} \cdot \alpha(4)^{\log r}.$$
As $\log\pbra{\frac{4}{r}} =\Theta(1)$, $\log r = \log\pbra{1 +
\Theta\pbra{\sqrt{\delta}}} = \Theta\pbra{\sqrt{\delta}}$, and $\alpha(4)
\leq M$, we get:
$$1 \leq \alpha(1)^{\Theta(1)}\cdot M^{\Theta\pbra{\sqrt{\delta}}}, \quad \quad
\text{and hence} \quad \quad
M^{-\Theta\pbra{\sqrt{\delta}}}\leq\alpha(1).$$
By (i) and the definition of $\alpha$, we have:
$$\sup_{t\in[\delta,1)} q(t) \geq M^{-\Theta\pbra{\sqrt{\delta}}}  \quad \quad
\text{and hence} \quad \quad
\sup_{t\in[0, 1]} p(t) \geq \sup_{\delta\in[0,1]} \delta
M^{-\Theta\pbra{\sqrt{\delta}}}.$$

Setting $\delta = \frac{\Theta(1)}{\log^2 M}$, we get that
	$$\sup_{t \in [0, 1]} |p(t)| \geq \frac{\Theta(1)}{\log^2 M}, $$
	and the lemma is proved.
\end{proof}

It is natural to wonder whether~\Cref{lem:main-lem} is quantitatively tight. The polynomial $p(t) = t(1-t)^{\log M}$ is easily seen to have length $M$ and  $\sup_{t \in [0,1]} p(t) =  \Theta(1/\log M)$, and it is tempting to wonder whether this might be the smallest achievable value.
However, it turns out that the $1/\log^2 M$ dependence of~\Cref{lem:main-lem} is in fact the best possible result, as shown by the following claim which we prove in \Cref{ap:best-possible}:
\medskip

\begin{claim} \label{claim:best-possible}
For sufficiently large $M$, there exists a real polynomial $p(t) = \sum_{i=1}^d c_it^i$ with $c_1 = 1$ and
$\sum_{i=1}^d |c_i| \leq M$ such that $$\sup_{t\in[0,1]} p(t)
\leq O\pbra{\pbra{\frac{1}{\log M}}^2}.$$
\end{claim}


\section{A General Approach to Quantitative Correlation Inequalities} \label{sec:general-approach}

This section presents our general approach to obtaining \emph{quantitative} correlation inequalities from \emph{qualitative} correlation inequalities. While our main result, \Cref{thm:main-thm}, is stated in an abstract setting, subsequent sections will instantiate this result in concrete settings that provided the initial impetus for this work.  \Cref{sec:application-royen} deals with the setting of centrally symmetric, convex sets over Gaussian space, \Cref{sec:application-talagrand} deals with finite product domains, and \Cref{sec:talagrand-solid-cube} deals with the solid cube $[-1,1]^n$.

\begin{theorem}[Main Theorem] \label{thm:main-thm}
	Consider a symmetric Markov semigroup $\left(\operatorP_t\right)_{t\geq0}$ on $L^2\left(\Omega^n, \Pi^{\otimes n}\right)$ with a chaos decomposition
	$$L^2\left(\Omega^n, \Pi^{\otimes n}\right) = \bigoplus_\ell \calW_\ell.$$
	Let $\left(\operatorP_t\right)_{t\geq0}$ be monotone compatible with $\calF\subseteq L^2\left(\Omega^n, \Pi^{\otimes n}\right)$,  where $\|f\| \leq 1$ for all $f\in\calF$. Furthermore, suppose that there exists $j^*\in\N_{>0}$ such that every $f \in \calF$ has a decomposition as
	\[
	f = \bigoplus_{\ell=0}^\infty f_{\ell \cdot j^*},
	\]
	i.e.~$f_\ell = 0$ for $j^\ast \nmid \ell$. Then for all $f, g\in\calF$, we have
	\begin{equation} \label{eq:main}
	\la f, g \ra - f_0 g_0 \geq {\frac 1 C} \cdot \Phi\left( \la f_{j^*}, g_{j^*} \ra \right),
	\end{equation}
	where recall from \Cref{eq:Phi} that $\Phi: [0,1] \to [0,1]$ is $\Phi(x) = \min\left\{x,\frac{x}{\log^2\left(1/x\right)}\right\}$ and $C>0$ is a universal constant.
\end{theorem}
 
 The proof of the above theorem uses an interpolating argument along the Markov semigroup, and appeals to \Cref{lem:main-lem} to obtain the lower bound. 

\begin{proof}[Proof of \Cref{thm:main-thm}]
	Fix $f, g\in\calF$ and let us write $a_\ell := \la f_\ell, g_\ell\ra$. 
	It follows from \Cref{def:chaos-decomposition} that $f_\ell, g_\ell$ are eigenfunctions of $\operatorP_t$ with eigenvalue $\lambda^\ell_t$. This, together with the assumption that $f = \oplus_{j^* \mid \ell} f_{\ell}$ and $g = \oplus_{j^* \mid \ell} g_{\ell},$ implies that for $t > 0$ we have
	\begin{equation} \label{eq:inner-product}
	\la \operatorP_t f, g\ra = \sum_{j^\ast \mid \ell} \lambda_t^\ell \la f_\ell, g_\ell \ra = \sum_{j^\ast \mid \ell} a_\ell \lambda_t^\ell.
	\end{equation}
	
Here we remark that 	the
argument to $\Phi(\cdot)$ in the right hand side of 
\Cref{eq:main}
 is non-negative, i.e. $a_{j^\ast} \geq 0$. To see this, observe that 
$$a_{j^\ast} = \frac{\partial}{\partial \lambda_t^{j^\ast}}\la \operatorP_t f, g\ra = \frac{\partial}{\partial t}\la \operatorP_t f, g\ra \cdot \frac{\partial t}{\partial \lambda_t^{j^\ast}} \geq 0$$  
where we used the monotone compatibility of $\calF$ with $(\operatorP_t)_{t\geq0}$ and Property 3 of \Cref{def:chaos-decomposition}.

Returning to \Cref{eq:inner-product}, rearranging terms gives that
	\begin{equation} \label{eqn:noise-poly}
		\la \operatorP_t f, g\ra - f_0g_0 = \sum_{\substack{\ell > 0 \\ j^* \mid \ell}} a_\ell \lambda_t^\ell = a_{j^*} p(\lambda^{j^*}_t), 
		\quad \quad \text{where} \quad \quad p(\lambda^{j^*}_t) := \lambda_t^{j^*} + \frac{1}{a_{j^*}}\sum_{\substack{\ell > {j^*} \\ {j^*} \mid \ell}} a_\ell\lambda_t^\ell.
	\end{equation}
	As $\lambda_t \in[0,1]$, we re-parametrize $u := \lambda_t^{j^*}$ and write $b_\ell := \frac{a_{\ell{j^*}}}{a_{j^*}}$ for ease of notation; this gives us $$p(u) = u + \sum_{\ell \geq 2} b_\ell u^\ell.$$ 
	By the Cauchy--Schwarz inequality, we have
	$$a_\ell^2 = \la f_\ell, g_\ell \ra^2 \leq \la f_\ell, f_\ell \ra \la g_\ell, g_\ell\ra = \|f_\ell\|^2 \|g_\ell\|^2,
	\quad \quad \text{and hence} \quad \quad |a_\ell| \leq \|f_\ell\|\|g_\ell\|.$$
	Once again using the Cauchy--Schwarz inequality, we get
	$$\sum_{\ell} |a_\ell| \leq  \sum_{\ell=0} \|f_\ell\|\cdot\|g_\ell\| \leq \sqrt{\left(\sum_\ell \|f_\ell\|^2\right)\cdot\left(\sum_\ell\|g_\ell\|^2\right)} \leq 1$$
	where the last inequality follows from the assumption that $\|f\| \leq 1$ for all $f\in \calF$. This implies that
	$$\sum_{\ell} \abs{b_\ell} = \frac{1}{|a_{j^*}|} \sum_\ell \left|a_{\ell \cdot {j^*}}\right| \leq \frac{1}{|a_{j^*}|} = \frac{1}{a_{j^*}}.$$
where the last equality holds because of $a_{j^*}\geq 0$ as shown earlier.
If $a_{j^*} > 2/3$ then $\sum_{\ell \geq 2} |b_i| \leq 1/2$ while $b_1=1$, from which it easily follows that $\sup_{u\in[0,1]} p(u)  \geq 1/2.$ If $a_{j^*} < 2/3$ then the power series $p(u)$ satisfies the assumptions of \Cref{lem:main-lem} with $M = \frac{1}{a_{j^*}}$. 
This gives us
	$$\sup_{u\in[0,1]} p(u) \geq \min\left\{{\frac 1 2},\Theta\left(\frac{1}{\log^2\left(a_{j^*}^{-1}\right)}\right)\right\}. $$
	It follows from \Cref{def:chaos-decomposition} that as $t$ ranges over $(0, \infty)$, $\lambda_t$ and consequently $u$ ranges over the interval $(0, 1]$. Together with \Cref{eqn:noise-poly}, this implies that
	$$\sup_{t \in (0,\infty)} \la \operatorP_t f, g\ra - f_0g_0 = \sup_{t\in(0,\infty)} a_{j^*}\cdot p(\lambda_t) = a_{j^*} \cdot \sup_{u \in (0,1]} p(u) \geq  \Theta\left(\min\left\{a_{j^*}, \frac{a_{j^*}}{\log^2\left(a_{j^*}^{-1}\right)}\right\}\right).$$
	However, because of monotone compatibility, we have that $\la \operatorP_t f, g\ra$ is decreasing in $t$. As $\operatorP_0 = \mathrm{Id}$, we can conclude that 
	$$\la f, g\ra - f_0 g_0 \geq \Theta\left(\min\left\{a_{j^*}, \frac{a_{j^*}}{\log^2\left(a_{j^*}^{-1}\right)}\right\}\right),
	$$
	which completes the proof.
\end{proof}


\section{Robust Correlation Inequalities over Gaussian Space} \label{sec:application-royen}

In this section we prove quantitative versions of Royen's Gaussian Correlation Inequality (GCI) \cite{roy14} for symmetric convex sets and Hu's inequality for symmetric convex functions\footnote{Note that the 0/1 indicator function of a convex set is not a convex function.}  \cite{hu1997ito}. We start by recalling some elementary facts about harmonic analysis over Gaussian space, after which we derive our ``robust'' form of the Gaussian Correlation Inequality in \Cref{subsec:robust-gci} as a consequence of \Cref{thm:main-thm}. In \Cref{subsec:sales-pitch-influences} we discuss how our robust GCI can be viewed as a Gaussian-space analogue of Talagrand's celebrated correlation inequality for monotone Boolean functions over the Boolean hypercube \cite{Talagrand:96}. We analyze the tightness of our robust GCI in \Cref{subsec:tightness-robust-gci}, and give a natural extension to quasiconcave functions over Gaussian space in \Cref{subsec:extension-gci-quasiconcave}. In \Cref{subsec:application-hu}, we state and prove our quantitative version of Hu's correlation inequality for symmetric convex functions over Gaussian space.

\subsection{Harmonic (Hermite) Analysis over Gaussian space} \label{subsec:hermite-anal}
Our notation and terminology presented in this subsection follows Chapter~11 of \cite{od2014analysis}. We say that an $n$-dimensional \emph{multi-index} is a tuple $\alpha \in \N^n$, and we define 
\begin{equation} \label{eq:index-notation}
\supp(\alpha) := \{i: \alpha_i \neq 0\},
\quad \quad
\#\alpha := |\supp(\alpha)|,
\quad \quad
|\alpha| := \sum_{i=1}^n \alpha_i.
\end{equation}

We write $\calN(0, 1)^n$ to denote the $n$-dimensional standard Gaussian distribution.
For $n \in \N_{>0}$, we write $L^2(\R^n, \gamma)$ to denote the space of functions $f: \R^n \to \R$ that have finite $2^\text{nd}$ moment $\|f\|_2^2$ under the standard Gaussian measure $\gamma$, that is: 
$$\|f\|_2^2 = \E_{\bz \sim \calN(0, 1)^n} \left[f(\bz)^2\right]^{1/2} < \infty.$$
We view $L^2(\R^n, \gamma)$ as an inner product space with $\la f, g \ra := \E_{\bz \sim \calN(0,1)^n}[f(\bz)g(\bz)]$ for $f, g \in L^2(\R^n, \gamma)$. We recall the ``Hermite basis'' for $L^2(\R, \gamma)$:

\begin{definition}[Hermite basis]
	The \emph{Hermite polynomials} $(h_j)_{j\in\N}$ are the univariate polynomials defined as
	$$h_j(x) = \frac{(-1)^j}{\sqrt{j!}} \exp\left(\frac{x^2}{2}\right) \cdot \frac{d^j}{d x^j} \exp\left(-\frac{x^2}{2}\right).$$
\end{definition}

\begin{fact} [Proposition~11.33, \cite{od2014analysis}] \label{fact:hermite-orthonormality}
	The Hermite polynomials $(h_j)_{j\in\N}$ form a complete, orthonormal basis for $L^2(\R, \gamma)$. For $n > 1$ the collection of $n$-variate polynomials given by $(h_\alpha)_{\alpha\in\N^n}$ where
	$$h_\alpha(x) := \prod_{i=1}^n h_{\alpha_i}(x)$$
	forms a complete, orthonormal basis for $L^2(\R^n, \gamma)$. 
\end{fact}

Given a function $f \in L^2(\R^n, \gamma)$ and $\alpha \in \N^n$, we define its \emph{Hermite coefficient on} $\alpha$ as $\widetilde{f}(\alpha) = \la f, h_\alpha \ra$. It follows that $f$ is uniquely expressible as $f = \sum_{\alpha\in\N^n} \widetilde{f}(\alpha)h_\alpha$ with the equality holding in $L^2(\R^n, \gamma)$; we will refer to this expansion as the \emph{Hermite expansion} of $f$. One can check that Parseval's and Plancharel's identities hold in this setting.

\begin{fact}[Plancharel's identity] \label{fact:hermite-plancharel}
	For $f, g \in L^2(\R^n, \gamma)$, we have:
	$$\la f, g\ra = \E_{\bz\sim\calN(0,1)^n}[f(\bz)g(\bz)] = \sum_{\alpha\in \N^n}\widetilde{f}(\alpha)\widetilde{g}(\alpha),$$
	and as a special case we have Parseval's identity,
	$$\la f, f\ra = \E_{\bz\sim\calN(0,1)^n}[f(\bz)^2] = \sum_{\alpha\in \N^n}\widetilde{f}(\alpha)^2.$$ 
\end{fact}

Next we recall the standard Gaussian noise operator (parameterized so that the noise rate $\rho$ ranges over $[0,1]$):

\begin{definition} [Ornstein-Uhlenbeck semigroup] \label{def:ou-operator}
	We define the \emph{Ornstein-Uhlenbeck semigroup} as the family of operators $(\U_\rho)_{\rho\in[0,1]}$ on the space of functions $f \in L^1(\R^n, \gamma)$ given by  
	$$\U_\rho f(x) := \E_{\bg \sim \calN(0, 1)^n}\left[f\left(\rho\cdot x + \sqrt{1-\rho}\cdot\bg\right)\right].$$
\end{definition}

The Ornstein-Uhlenbeck semigroup is sometimes referred to as the family of \emph{Gaussian noise operators} or \emph{Mehler transforms}. The Ornstein-Uhlenbeck semigroup acts on the Hermite expansion as follows:

\begin{fact} [Proposition~11.33, \cite{od2014analysis}] \label{fact:gaussian-noise-expansion}
	For $f \in L^2(\R^n, \gamma)$, the function $\U_\rho f$ has Hermite expansion
	$$\U_\rho f = \sum_{\alpha\in\N^n}\rho^{|\alpha|}\widetilde{f}(\alpha)h_\alpha.$$
\end{fact}

\subsection{A Robust Extension of the Gaussian Correlation Inequality} \label{subsec:robust-gci}


We start by making a crucial observation regarding Royen's proof of the Gaussian correlation inequality (GCI) \cite{roy14}. Recall that the GCI states that if $K$ and $L$ are the indicator functions of two centrally symmetric (i.e. $K(x) = 1$ implies $K(-x) = 1$), convex sets, then they are non-negatively correlated under the Gaussian measure; that is,
$$\E_{\bx\sim\calN(0,1)^n}[K(\bx)L(\bx)] - \E_{\bx\sim\calN(0,1)^n}[K(\bx)]\E_{\by\sim\calN(0,1)^n}[K(\by)]\geq 0.$$

In order to prove this, Royen interpolates between $\E[K]\E[L]$ and $\E[KL]$ via the Ornstein-Uhlenbeck semigroup, and shows that this interpolation is monotone nondecreasing; indeed, note that
$$\la \U_{1} K, L \ra = \E_{\bx\sim\calN(0,1)^n}[K(\bx)L(\bx)], \qquad \text{and that} \qquad \la \U_{0}K, L\ra \E_{\bx\sim\calN(0,1)^n}[K(\bx)]\E_{\by\sim\calN(0,1)^n}[K(\by)].$$

Thus, Royen's main result can be interpreted as follows (we refer the interested reader to a simplified exposition of Royen's proof by Lata\l a and Matlak \cite{latala-matlak-gci} for further details):

\begin{proposition}[Royen's Theorem, \cite{roy14}] \label{prop:csc-sets-monotone-compatible}
	Let $\calF_\mathrm{csc} \subseteq L^2\left(\R^n, \gamma\right)$ be the family of indicators of centrally symmetric, convex sets, and let $(\U_\rho)_{\rho\in[0,1]}$ be the Ornstein-Uhlenbeck semigroup. Then for $K, L \in \calF_\mathrm{csc}$, we have
	$$\frac{\partial}{\partial\rho} \la \U_\rho K, L\ra \geq 0 \quad \quad \text{for all~}0<\rho<1.$$
	In particular, $\calF_\mathrm{csc}$ is monotone compatible with $(\U_\rho)_{\rho\in[0,1]}$. 
\end{proposition}

Recall that we are parametrizing the Ornstein-Uhlenbeck semigroup by $\rho\in[0, 1]$ where $\rho = e^{-t}$ for $t\in[0, \infty)$; see the discussion following \Cref{def:markov-semigroup}. We can now state our main result:

\begin{theorem}[Quantitative GCI]	\label{thm:robust-gci}
	Let $\calF_\mathrm{csc} \subseteq L^2\left(\R^n, \gamma\right)$ be the family of indicators of centrally symmetric, convex sets. Then for $K, L \in \calF_\mathrm{csc}$, we have
	$$\E[KL] - \E[K]\E[L] \geq {\frac 1 C} \cdot \Phi\left(\sum_{|\alpha| = 2} \widetilde{K}(\alpha)\widetilde{L}(\alpha)\right)$$
where recall from \Cref{eq:Phi} that $\Phi: [0,1] \to [0,1]$ is $\Phi(x) = \min\left\{x,\frac{x}{\log^2\left(1/x\right)}\right\}$ and $C>0$ is a universal constant.
\end{theorem}

\begin{proof}
	Consider the orthogonal decomposition 
	$$L^2(\R^n, \gamma) = \bigoplus_{i=0}^\infty \calW_i$$
	where $\calW_i = \mathrm{span}\left\{h_\alpha : |\alpha| = i\right\}$; the orthogonality of this decomposition follows from \Cref{fact:hermite-orthonormality}. From \Cref{fact:gaussian-noise-expansion}, it follows that this decomposition is in fact a chaos decomposition (recall \Cref{def:chaos-decomposition}) with respect to the Ornstein-Uhlenbeck semigroup $(\U_\rho)_{\rho\in[0,1]}$. 
	
	If $K \in \Fcsc$, then $K(x) = K(-x)$ as $K$ is the indicator of a centrally symmetric set; in other words, $K$ is an even function. Consequently, its Hermite expansion is given by 
	$$K = \bigoplus_{\substack{i = 0 \\ |\alpha| = 2i}}^\infty h_\alpha.$$
	Furthermore, from \Cref{fact:hermite-plancharel}, we have that
	$$\|K\|^2 = \sum_{\alpha\in\N^n} \widetilde{K}(\alpha)^2 = \E\left[K^2\right] \leq 1.$$
	It follows that the hypotheses of \Cref{thm:main-thm} hold for $\Fcsc$ with $j^* = 2$; consequently, for $K, L \in \Fcsc$ we have
	$$\la \U_1 K, L \ra - \la \U_0 K, L \ra = \E[KL] - \E[K]\E[L] \geq {\frac 1 C} \cdot \Phi\left(\sum_{|\alpha| = 2} \widetilde{K}(\alpha)\widetilde{L}(\alpha)\right),$$
which completes the proof of the theorem. 
\end{proof}

\subsection{Interpreting \Cref{thm:robust-gci}} \label{subsec:sales-pitch-influences}

Recall Talagrand's correlation inequality \cite{Talagrand:96}: If $f, g: \zo^n \to \zo$ are monotone Boolean functions, then 
$$\E[fg] - \E[f]\E[g] \geq {\frac 1 C} \cdot \Psi\pbra{\sum_{i=1}^n \widehat{f}(i)\widehat{g}(i)}$$
where $\Psi(x) = \frac{x}{\log(e/x)}$. However (see Chapter 2 of \cite{od2014analysis}), for monotone $f: \bits^n \to \bits$, we have $\widehat{f}(i) = \Inf_{i}[f]$ where $$\Inf_i[f] := \Pr_{\bx\sim\bits^n}\sbra{f\pbra{\bx} \neq f\pbra{\bx^{\oplus i}}}.$$ In other words, the degree-1 Fourier coefficient $\widehat{f}(i)$ captures the ``dependence'' of $f$ on its $i^\text{th}$ coordinate, and the quantity $\sum_{i=1}^n\widehat{f}(i)\widehat{g}(i)$ captures the extent to which ``both $f$ and $g$ simultaneously depend on the same coordinates''. This intuitively explains why it is plausible for such a quantity to appear in Talagrand's inequality.

Inspired by the resemblance between our quantitative Gaussian correlation inequality and Talagrand's correlation inequality, we believe that the (negated) degree-2 Hermite coefficients of centrally symmetric, convex sets over Gaussian space are natural analogues of the degree-1 Fourier coefficients (i.e.~the coordinate influences) of monotone Boolean functions. However, while functions on the Boolean hypercube have influences only along $n$ ``directions'', there are infinitely many directions over Gaussian space. We make the following definition:

\begin{definition}[Influences for $\Fcsc$] \label{def:csc-influence}
	Let $K \subseteq \R^n$ be a centrally symmetric, convex set. Given a unit vector $v\in  S^{n-1}$, we define the \emph{influence of $K$ along direction $v$} as
	$$\Inf_v[K] := -\widetilde{K}(2v) = \E_{\bx\sim\calN(0,1)^n}\sbra{-K(\bx)h_2(v\cdot\bx)}$$
	where $h_2(x) = \frac{x^2 - 1}{\sqrt{2}}$ is the degree-2 univariate Hermite polynomial (see Section 11.2 of \cite{od2014analysis}). 
\end{definition}

%

It follows from the proof of \Cref{thm:main-thm} that the quantity $\sum_{|\alpha| = 2} \widetilde{K}(\alpha)\widetilde{L}(\alpha)$ for $K, L \in \Fcsc$ is non-negative. In fact more is true: if $K$ is a centrally symmetric, convex set, then each $\Inf_{e_i}[K]$ is itself non-negative. The proof of the following proposition is deferred to \Cref{ap:influences-nonneg}:

\begin{proposition}[Influences are non-negative] \label{lem:influence-basics}
	If $K$ is a centrally symmetric, convex set, then $\Inf_v[K] \geq 0$ for all $v\in S^{n-1}$, with equality holding if and only if $K(x) = K(y)$ whenever $x_{v^\perp} = y_{v^\perp}$ (the projection of $x$ orthogonal to $v$ coincides with that of $y$), except  possibly on a set of measure zero.
\end{proposition}

It is natural to define the ``total influence of $K$'' to be $\Inf[K] := \sum_{i=1}^n \Inf_{e_i}[K]$; we observe that this quantity is given by 
$$\Inf[K] = -\sum_{i=1}^n \widetilde{K}(2e_i) = \frac{1}{\sqrt{2}}\E_{\bx\sim\calN^(0,1)^n}\left[-f(\bx)\sum_{i=1}^n(x_i^2 - 1)\right] = \frac{1}{\sqrt{2}}\E_{\bx\sim\calN(0,1)^n}\left[-f(\bx)\cdot(\|\bx\|^2 - n)\right],$$
and hence it is invariant under orthogonal transformations (i.e.~any orthonormal basis $v_1,\dots,v_n$ could have been used in place of $e_1,\dots,e_n$ in defining $\Inf[K]$).

The above discussion suggests that the notion of ``influences'' for centrally symmetric, convex sets in Gaussian space proposed in \Cref{def:csc-influence} is indeed ``influence-like''. As mentioned in \Cref{sec:intro}, a forthcoming paper \cite{Influence-future} will further explore this notion.

\subsection{On the Tightness of \Cref{thm:robust-gci}} \label{subsec:tightness-robust-gci}

In \cite{Talagrand:96}, Talagrand gave the following family of example functions for which \Cref{eq:Talagrand} is tight up to constant factors:  let $f, g : \{0,1\}^n\to\{0, 1\}$ be given by 
$$f(x) = \begin{cases} 1 & \sum_i x_i \geq n - k \\ 0 & \text{otherwise} \end{cases},
\qquad \text{and} \qquad 
g(x) = \begin{cases} 1 & \sum_i x_i > k \\ 0 & \text{otherwise} \end{cases}$$
where $k \leq n/2$. Writing $\eps$ to denote $\E[f]$, we have $\eps^2 = \eps - \eps(1-\eps) = \E[fg]-\E[f]\E[g]$, and it can be shown that $\Psi\left(\sum_{i=1}^n \widehat{f}(i)\widehat{g}(i)\right) = \Theta(\eps^2)$, so \Cref{eq:Talagrand} is tight up to constant factors.  We note that in this example $f$ and $g$ are the indicator functions of Hamming balls, and that $f \subseteq g$ (i.e. $f(x) = 1$ implies that $g(x) = 1$). Motivated by this example, we consider an analogous pair of functions in the setting of centrally symmetric, convex sets over Gaussian space, where we use origin-centered balls of different radii in place of Hamming balls. The main result of this subsection is that such an example witnesses that  \Cref{thm:robust-gci} can be tight up to a logarithmic factor (corresponding to the log factor difference between $\Phi$ and $\Psi$). In what follows, all expectations and probabilities are with respect to the $n$-dimensional Gaussian measure. As before, we will identify centrally symmetric, convex sets with their indicator functions.

Let $K, L  \in \Fcsc$ be $n$-dimensional origin-centered balls of radii $r_1$ and $r_2$ respectively such that $r_1 < r_2$,  $\E[K] = \eps$, and $\E[L]=1-\eps.$  As $K \subseteq L$, we have $\E[KL] - \E[K]\E[L] = \eps - \eps(1-\eps) = \eps^2.$ Since $K(x_1,\dots,x_n) = K(x_1,\dots,x_{i-1},-x_i,x_{i+1},\dots,x_n)$ for all $x \in \R^n$ and all $i \in [n]$, it easily follows that $\widetilde{K}(e_i + e_j) = \E[K(\bx) \bx_i \bx_j] = 0$ for all $i \neq j$, and the same is true for $L$.
It follows that
\[
\sum_{|\alpha|=2} \widetilde{K}(\alpha) = 
\sum_{i=1}^n \widetilde{K}(2e_i)
\]
and similarly for $L$.
Furthermore, as $K, L$ are rotationally invariant, we have $\widetilde{K}(2e_i) = \widetilde{K}(2e_j)$ and $\widetilde{L}(2e_i) = \widetilde{L}(2e_j)$ for all $1\leq i, j \leq n$. By definition, we have
$$ - \widetilde{K}(2e_i) = \la K, -h_2(x_i)\ra = \E_{\bx\sim\calN(0,1)^n}\left[K(\bx)\frac{\left(1 - \bx_i^2\right)}{\sqrt{2}}\right]$$
as $h_2(x) = \frac{x^2 - 1}{\sqrt{2}}$. Now, note that
\begin{align*}
	-\sum_{i=1}^n \widetilde{K}(2e_i) &= \frac{1}{\sqrt{2}}\sum_{i=1}^n \E_{\bx\sim\calN(0,1)^n}\left[K(\bx)\left(1- \bx_i^2\right)\right]\\
	&= \frac{1}{\sqrt{2}} \E_{\bx\sim\calN(0,1)^n}\left[K(\bx)\left(\sum_{i=1}^n 1 - \bx_i^2\right)\right] \\
	&= \frac{1}{\sqrt{2}} \E_{\bx\sim\calN(0,1)^n}\left[K(\bx)\left(n - \|\bx\|^2\right)\right].
\end{align*}

In order to obtain a lower bound on the above quantity, we will show that $(n - \|\bx\|^2)$ is ``large'' with non-trivial probability for $\bx \in K$; we will do so by approximating $(n - \|\bx\|^2)$ by a Gaussian distribution, and then appealing to the Berry-Esseen Central Limit Theorem (see \cite{berry, esseen} or, for example, Section~11.5 of \cite{od2014analysis}).  By the Berry-Esseen theorem, we have that for $t \in \R$,
\begin{equation} \label{eq:berry-essen-approx}
\left|\Pr_{\bx\sim\calN(0,1)^n}\sbra{\frac{\|\bx\|^2 - n}{\sqrt{n}} \leq t} - \Pr_{\by \sim\calN(0,1)}\sbra{\by \leq t}\right| \leq {\frac{c_1}{\sqrt{n}}}	
\end{equation}
for some absolute constant $c_1$. We assume that $\eps \gg c_1/\sqrt{n}$. By standard anti-concentration of the lower tail of the Gaussian distribution, we have that $\Pr_{\by\sim\calN(0,1)}\sbra{\by \leq t} \geq \frac{\eps}{2}$ for $t = - c_2\sqrt{\ln\pbra{\frac{2}{\eps}}}$ where $c_2$ is an absolute constant. Then it follows from \Cref{eq:berry-essen-approx} that 
$$\Pr_{\bx\sim\calN(0,1)^n}\sbra{\frac{\|\bx\|^2 - n}{\sqrt{n}} \leq -c_2\sqrt{\ln\pbra{\frac{2}{\eps}}}} \geq \frac{\eps}{2} \pm \frac{c_1}{\sqrt{n}} \gtrsim \frac{\eps}{2}$$
which can be rewritten as
$$\Pr_{\bx\sim\calN(0,1)^n}\sbra{\|\bx\|^2 \leq n - c_2\sqrt{n\ln\pbra{\frac{2}{\epsilon}}}} \gtrsim \frac{\eps}{2}.$$
As $\E[K] = \epsilon$, it follows that
$$\E_{\bx\sim\calN(0,1)^n}\sbra{K(\bx)\pbra{n - \|\bx\|^2}} = \Omega\pbra{\epsilon \sqrt{n\ln\pbra{\frac{2}{\eps}}}}$$
from which we have $-\widetilde{K}(2e_i) \geq \Omega\pbra{\eps\sqrt{\frac{1}{n}\ln\pbra{\frac{2}{\eps}}}}$ for all $i \in [n]$. A similar calculation for $L$ gives that $-\wt{L}(2e_i) \geq \Omega\pbra{\eps\sqrt{\frac{1}{n}\ln\pbra{\frac{2}{\eps}}}}$, from which it follows that $\sum_{i=1}^n \wt{K}(2e_i)\wt{L}(2e_i) = \Omega\pbra{\eps^2 \ln\pbra{\frac{2}{\eps}}}$. 
Recalling \Cref{eq:Phi}, we get that for small enough $\eps$, the quantity
$$\Phi\pbra{\sum_{|\alpha|=2} \wt{K}(\alpha)\wt{L}(\alpha)} = \Omega\left({\frac {\eps^2}{\log(2/\eps)}}\right),$$
which lets us conclude that \Cref{thm:robust-gci} is tight to within a logarithmic factor.

\subsection{Extension to Centrally Symmetric, Quasiconcave Functions} \label{subsec:extension-gci-quasiconcave}

It is natural to ask whether \Cref{thm:robust-gci} can be extended to a broader class of functions than 0/1-valued indicator functions of centrally symmetric, convex sets $\Fcsc$. Indeed, the GCI implies the monotone compatibility of  centrally symmetric, \emph{quasiconcave} (see \Cref{def:quasiconcave-func}),  non-negative functions (which is a larger family of functions than $\Fcsc$) with the Ornstein-Uhlenbeck semigroup. This allows us to once again use \Cref{thm:main-thm} to obtain a quantitative correlation inequality for this family of functions.

\begin{definition}[Quasiconcave function] \label{def:quasiconcave-func}
A function $f: \R^n \to \R$ is \emph{quasiconcave} if for all $\lambda \in [0, 1]$ we have
$$f\left(\lambda x + (1-\lambda)y\right) \geq \min\cbra{f(x), f(y)}.$$
\end{definition}

It is easy to check that all concave and log-concave functions are quasiconcave. Moreover, 0/1-valued indicators of convex sets are also quasiconcave. We recall the following alternative characterization of quasiconcave functions:

\begin{fact} \label{fact:quasiconcave-upper-level-sets}
	A function $f : \R^n \to \R$ is quasiconcave if and only if its upper-level sets are convex, i.e. for all $t \in \R$, the set $\{x \in \R^n : f(x) \geq t \}$ is convex. 
\end{fact}

It follows from \Cref{fact:quasiconcave-upper-level-sets} that any centrally symmetric, quasiconcave function $f: \R^n \to \R_{\geq 0}$ can be expressed as an integral over indicators of centrally symmetric, convex sets:
\begin{equation} \label{eq:csq-indicator-int}
	f(x) = \int_{0}^\infty \mathbf{1}_{\left[f(x) \geq t\right]} dt.
\end{equation}
We will now show that the family of centrally symmetric, quasiconcave functions is monotone compatible with the Ornstein--Uhlenbeck semigroup.

\begin{proposition}[Monotone compatibility of $\Fcsq$] \label{prop:monotone-compat-quasiconcave}
	Let $\Fcsq \subseteq L^2(\R^n, \gamma)$ be the family of  centrally symmetric, quasiconcave functions taking values in $\R_{\geq 0}$. Then for $f, g \in \Fcsq$, we have
	$\frac{\partial}{\partial\rho}\la \U_\rho f, g\ra \geq 0.$
\end{proposition}

\begin{proof}
	Expressing $f, g$ as integrals over indicator functions as in \Cref{eq:csq-indicator-int}, we get
	$$f(x) = \int_{0}^\infty \mathbf{1}_{[f(x) \geq t]} dt \qquad \text{and} \qquad g(x) = \int_{0}^\infty \mathbf{1}_{[g(x) \geq t]} dt.$$
Using Fubini's theorem to commute integration and expectation, we get
	\begin{align*}
		\U_\rho f(x) &= \E_{\bz}\sbra{f\pbra{\rho x + \sqrt{1-\rho^2}\bz}}
		= \E_{\bz}\sbra{\int_{0}^\infty \mathbf{1}_{\sbra{f\pbra{\rho x + \sqrt{1-\rho^2}\bz} \geq t}}dt} 
		= \int_0^\infty \Pr_{\by\sim N_\rho(x)}\sbra{f(\by) \geq t} dt.
	\end{align*}	
	where the expectations are with respect to $\bz \sim \calN(0,1)^n$.
Fubini's theorem again gives
	\begin{align*}
		\frac{\partial}{\partial\rho}\la \U_\rho f, g\ra &= \frac{\partial}{\partial\rho} \E_{\bx\sim\calN(0,1)^n}\sbra{\left(\int_{0}^\infty  \mathbf{1}_{[\U_\rho f(\bx) \geq t]}dt\right) \left(\int_{0}^\infty \mathbf{1}_{[g(\bx) \geq s]} ds\right)} \\
		&= \frac{\partial}{\partial\rho}\int_{0}^\infty\int_{0}^\infty  \Pr_{\substack{\bx\sim\calN(0,1)^n\\\by\sim N_\rho(\bx)}}\sbra{f(\by) \geq t, g(\bx) \geq s} \cdot dt\cdot ds\\
		&= \frac{\partial}{\partial\rho} \int_0^\infty \int_0^\infty  \abra{\U_\rho f_t, g_s} \cdot dt\cdot ds
	\end{align*}
where $f_t : \R^n \to \zo$ is the indicator of the convex set $\{ x \in \R^n : f(x) \geq t\}$ and $g_s$ is defined similarly; in particular, $f_s, f_t \in \Fcsc$ for all $s, t \in \R$. The Leibniz rule lets us commute integration and differentiation, so we get
$$\frac{\partial}{\partial\rho}\la \U_\rho f, g\ra = \int_0^\infty \int_0^\infty  \frac{\partial}{\partial\rho} \abra{\U_\rho f_t, g_s} \cdot dt\cdot ds \geq 0$$
where the final inequality follows from the monotone compatibility of $\Fcsc$ with $\U_\rho$. 
\end{proof}

We note that this immediately implies a \emph{qualitative} correlation inequality for $\Fcsq$, that is, for $f, g\in\Fcsq$, we have $\E[fg]-\E[f]\E[g] \geq 0.$ Using \Cref{thm:main-thm}, we can obtain the following \emph{quantitative} correlation inequality for $\Fcsq$; the proof of the following proposition is identical to that of \Cref{thm:robust-gci} and is therefore omitted. 

\begin{proposition} \label{prop:robust-correlation-quasiconcave}
Let $\Fcsq\subseteq L^2(\R^n, \gamma)$ as in \Cref{prop:monotone-compat-quasiconcave}. Then for all $f, g \in \Fcsq$, we have
$$\E[fg] - \E[f]\E[g] 
\geq {\frac 1 C} \cdot \Phi\left(\sum_{|\alpha| = 2} \widetilde{f}(\alpha)\widetilde{g}(\alpha)\right)$$
where recall from \Cref{eq:Phi} that $\Phi: [0,1] \to [0,1]$ is $\Phi(x) = \min\left\{x,\frac{x}{\log^2\left(1/x\right)}\right\}$ and $C>0$ is a universal constant.

\end{proposition}

\subsection{A Quantitative Extension of Hu's Inequality for Convex Functions}
\label{subsec:application-hu}

In this section, we consider the following special case of Hu's inequality \cite{hu1997ito}:

\begin{theorem}[Hu's inequality] \label{thm:hu-inequality}
	Let $f, g : \R^n \to \R$ be centrally symmetric, convex functions. Then
	$$\E_{\bx\sim\calN(0,1)^n}\sbra{f(\bx)g(\bx)} - \E_{\bx\sim\calN(0,1)^n}\sbra{f(\bx)}\E_{\by\sim\calN(0,1)^n}\sbra{g(\by)} \geq 0.$$
\end{theorem}

As in \Cref{subsec:robust-gci}, we will obtain a quantitative extension of \Cref{thm:hu-inequality} by appealing to \Cref{thm:main-thm}. The Markov semigroup we will use here will once again be the Ornstein--Uhlenbeck semigroup $(\U_\rho)_{\rho\in[0,1]}$---monotone compatibility of this semigroup with the family of centrally symmetric, convex functions (which we will denote $\Fcvx$) was proved by Harg\'e \cite{harge2005characterization}. 

\begin{fact}[Proof of Theorem 2.1, \cite{harge2005characterization}]
	Let $\Fcvx$ denote the family of centrally symmetric, convex functions with $\|f\| \leq 1$ for all $f\in\Fcvx$. Then $\Fcvx$ is monotone compatible with $(\U_\rho)_{\rho\in[0,1]}$. 
\end{fact}

The proof of the following result is identical to that of \Cref{thm:robust-gci} and is therefore omitted. 

\begin{theorem}[Quantitative Hu's inequality] \label{thm:our-hu}
	Let $\Fcvx \subseteq L^2\left(\R^n, \gamma\right)$ be the family of centrally symmetric, convex functions. Then for $f,g \in \calF_\mathrm{cvx}$, we have
	$$\E[fg] - \E[f]\E[g] \geq {\frac 1 C} \cdot \Phi\left(\sum_{|\alpha| = 2} \widetilde{f}(\alpha)\widetilde{g}(\alpha)\right)$$
where recall from \Cref{eq:Phi} that $\Phi: [0,1] \to [0,1]$ is $\Phi(x) = \min\left\{x,\frac{x}{\log^2\left(1/x\right)}\right\}$ and $C>0$ is a universal constant.
\end{theorem}


\section{A Quantitative Correlation Inequality for Arbitrary Finite Product Domains} \label{sec:application-talagrand}

The main result of this section, \Cref{thm:our-talagrand}, is an extension of Talagrand's correlation inequality \cite{Talagrand:96} to real-valued functions on general, finite, product spaces. (Recall that Talagrand's inequality applies only to Boolean-valued functions on the domain $\zo^n$ under the uniform distribution.) We start by briefly setting up harmonic analysis over finite product spaces, as well as recalling the \emph{Efron--Stein decomposition} which will be used to interpret our quantitative correlation inequalities in \Cref{subsec:efron-stein-interpretation}. 

\subsection{Harmonic Analysis over Finite Product Spaces} \label{subsec:efron-stein}

Our notation and terminology presented in this subsection follows Chapter~8 of \cite{od2014analysis}.
We use multi-index notation for $\alpha\in\N^n$ as defined in \Cref{eq:index-notation}.

Let $(\Omega, \pi)$ be a finite probability space with $|\Omega| = m \geq 2$, where we always assume that the distribution $\pi$ over $\Omega$ has full support (i.e. $\pi(\omega)>0$ for every $\omega \in \Omega$).
We write $L^2(\Omega^n,\pi^{\otimes n})$ for the real inner product space of functions $f: \Omega^n \to \R$, with inner product $\la f, g \ra := \E_{\bx \sim \pi^{\otimes n}}[f(\bx)g(\bx)].$

It is easy to see that there exists an orthonormal basis for the inner product space $L^2(\Omega,\pi)$, i.e.~a set of functions $\phi_0,\dots,\phi_{m-1}: \Omega \to \R$, with $\phi_0 = 1,$ that are orthonormal with respect to $\pi$. Moreover, such a basis extends to an orthonormal basis for $L^2(\Omega^n,\pi^{\otimes n})$ by a straightforward $n$-fold product construction: given a multi-index $\alpha \in \N^n_{<m}$, if we define $\phi_\alpha \in L^2(\Omega^n,\pi^{\otimes n})$ as
\[
\phi_\alpha(x) := \prod_{i=1}^n \phi_{\alpha_i}(x_i),
\]
then the collection $(\phi_\alpha)_{\alpha \in \N^n_{<m}}$ is an orthonormal basis for $L^2(\Omega^n,\pi^{\otimes n})$ (see Proposition~8.13 of \cite{od2014analysis}). So every function
$f: \Omega^n \to \R$ has a  decomposition
\begin{equation}~\label{eq:f-decomp-alpha}
f = \sum_{\alpha \in \N^n_{< m}} \widehat{f}(\alpha) \phi_{\alpha}.
\end{equation}
This can be thought of as a ``Fourier decomposition" for $f$, in that it  satisfies both Parseval's and Plancharel's identities (see Proposition 8.16 of \cite{od2014analysis}). We now proceed to define a noise operator for finite product spaces.

\begin{definition} [Noise operator for finite product spaces] \label{def:bonami-beckner-operator}
Fix a finite product probability space $L^2(\Omega^n,\pi^{\otimes n}).$ For $\rho \in [0, 1]$ we define the \emph{noise operator for $L^2(\Omega^n,\pi^{\otimes n})$} as the linear operator 
	$$\T_\rho f(x) := \E_{\by \sim N_\rho(x)}[f(\by)],
	$$
where ``$\by \sim N_\rho(x)$'' means that $\by \in \Omega^n$ is randomly chosen as follows: for each $i \in [n]$, with probability $\rho$ set $\by_i$ to be $x_i$ and with the remaining $1-\rho$ probability set $\by_i$ by independently making a draw from $\pi$.
\end{definition}

It is easy to check that\ignore{\rnote{fixed a typo; this had been ``$\T_\rho f = \sum_{\alpha} \widehat{f}(\alpha)^{\#\alpha}  \phi_\alpha$''}} $\T_\rho f = \sum_{\alpha} \rho^{\#\alpha} \widehat{f}(\alpha)\phi_\alpha$ (Proposition~8.28 of \cite{od2014analysis}). A crucial thing to note here is that for $m>2$, the choice of an orthonormal basis for the measure space $(\Omega, \pi)$ is not canonical. On the other hand,
the definition of  the noise operator $T_{\rho}$ does not depend on the choice of the basis for $(\Omega, \pi)$. Consequently, it should be possible to define the action of the operator $T_{\rho}$ on  $f$ without referencing the specific basis $(\phi_\alpha)_{\alpha \in \N^n_{<m}}$. 

Towards this, we now recall the so-called \emph{Efron-Stein decomposition} of functions in $L^2(\Omega^n,\pi^{\otimes n})$. We note that while the Efron-Stein decomposition is canonical, it is somewhat less explicit than the decomposition given in \eqref{eq:f-decomp-alpha}. 
\begin{theorem} [Efron-Stein decomposition, Theorem~8.35 of \cite{od2014analysis}] \label{thm:es}
Let $f \in L^2(\Omega^n,\pi^{\otimes n}).$ 
Then $f$ has a unique decomposition as
\[
f = \sum_{S \subseteq [n]} f^{=S}
\]
where the functions $f^{=S} \in L^2(\Omega^n,\pi^{\otimes n})$ satisfy the following:

\begin{enumerate}

\item $f^{=S}$ depends only on the coordinates in $S$.

\item The decomposition is orthogonal:  $\langle f^{=S},f^{=T} \rangle = 0$ for $S \neq T$.

\end{enumerate}
\end{theorem}

The noise operator interacts with the Efron-Stein decomposition in the following useful way:

\begin{fact} [Proposition~8.28 and Proposition~8.36, \cite{od2014analysis}] \label{fact:finite-noise-expansion}
	For $f \in L^2(\Omega^n, \pi^{\otimes n})$, the function $\T_\rho f$ has an orthogonal expansion as
	$$\T_\rho f = \sum_{S \subseteq [n]}\rho^{|S|}f^{=S}.$$
\end{fact}
While the Efron-Stein decomposition is canonical, in the next subsection, our arguments employ the (arbitrary but) fixed basis of $(\Omega, \pi)$ given by $\{\phi_i\}_{0 \le i\le m-1}$. We believe that referencing the basis explicitly makes the arguments more illuminating.


\ignore{
\subsection{Harmonic Analysis over Finite Product Spaces} \label{subsec:efron-stein}

Our notation and terminology presented in this subsection follows Chapter~8 of \cite{od2014analysis}.
We use multi-index notation for $\alpha\in\N^n$ as defined in \Cref{eq:index-notation}.

Let $(\Omega, \pi)$ be a finite probability space with $|\Omega| = m \geq 2$, where we always assume that the distribution $\pi$ over $\Omega$ has full support (i.e. $\pi(\omega)>0$ for every $\omega \in \Omega$).
We write $L^2(\Omega^n,\pi^{\otimes n})$ for the real inner product space of functions $f: \Omega^n \to \R$, with inner product $\la f, g \ra := \E_{\bx \sim \pi^{\otimes n}}[f(\bx)g(\bx)].$

We recall (Remark~8.9 of \cite{od2014analysis}) that there exists a ``Fourier basis'' for the inner product space $L^2(\Omega,\pi)$, i.e.~a set of functions $\phi_0,\dots,\phi_{m-1}: \Omega \to \R$, with $\phi_0 = 1,$ that are orthonormal with respect to $\pi$.  Moreover, such a basis extends to a Fourier basis for $L^2(\Omega^n,\pi^{\otimes n})$ by a straightforward $n$-fold product construction: given a multi-index $\alpha \in \N^n_{<m}$, if we define $\phi_\alpha \in L^2(\Omega^n,\pi^{\otimes n})$ as
\[
\phi_\alpha(x) := \prod_{i=1}^n \phi_{\alpha_i}(x_i),
\]
then the collection $(\phi_\alpha)_{\alpha \in \N^n_{<m}}$ is an orthonormal basis for $L^2(\Omega(^n,\pi^{\otimes n})$ (see Proposition~8.13 of \cite{od2014analysis}). So every function
$f: \Omega^n \to \R$ has a Fourier decomposition
\[
f = \sum_{\alpha \in \N^n_{< m}} \widehat{f}(\alpha) \phi_{\alpha}.
\]
This Fourier decomposition satisfies Parseval's and Plancharel's identities (see Proposition 8.16 of \cite{od2014analysis}). We now proceed to define a noise operator for finite product spaces. 

\begin{definition} [Noise operator for finite product spaces] \label{def:bonami-beckner-operator}
Fix a finite product probability space $L^2(\Omega^n,\pi^{\otimes n}).$ For $\rho \in [0, 1]$ we define the \emph{noise operator for $L^2(\Omega^n,\pi^{\otimes n})$} as the linear operator 
	$$\T_\rho f(x) := \E_{\by \sim N_\rho(x)}[f(\by)],
	$$
where ``$\by \sim N_\rho(x)$'' means that $\by \in \Omega^n$ is randomly chosen as follows: for each $i \in [n]$, with probability $\rho$ set $\by_i$ to be $x_i$ and with the remaining $1-\rho$ probability set $\by_i$ by independently making a draw from $\pi$.
\end{definition}

It is easy to check that $\T_\rho f = \sum_{\alpha} \widehat{f}(\alpha)^{\#\alpha}\phi_\alpha$ (Proposition~8.28 of \cite{od2014analysis}).
However, in general there is not a unique or canonical choice for a Fourier basis of $L^2(\Omega,\pi)$, and hence it is more natural to consider the \emph{Efron-Stein decomposition} of functions in $L^2(\Omega^n,\pi^{\otimes n})$. The main result we use about this decomposition is the following:

\begin{theorem} [Efron-Stein decomposition, Theorem~8.35 of \cite{od2014analysis}] \label{thm:es}
Let $f \in L^2(\Omega^n,\pi^{\otimes n}).$ 
Then $f$ has a unique decomposition as
\[
f = \sum_{S \subseteq [n]} f^{=S}
\]
where the functions $f^{=S} \in L^2(\Omega^n,\pi^{\otimes n})$ satisfy the following:

\begin{enumerate}

\item $f^{=S}$ depends only on the coordinates in $S$.

\item The decomposition is orthogonal:  $\langle f^{=S},f^{=T} \rangle = 0$ for $S \neq T$.

\end{enumerate}
\end{theorem}

The noise operator interacts with the Efron-Stein decomposition in the following useful way:

\begin{fact} [Proposition~8.28 and Proposition~8.36, \cite{od2014analysis}] \label{fact:finite-noise-expansion}
	For $f \in L^2(\Omega^n, \pi^{\otimes n})$, the function $\T_\rho f$ has an orthogonal expansion as
	$$\T_\rho f = \sum_{S \subseteq [n]}\rho^{|S|}f^{=S}.$$
\end{fact}

}
\subsection{The Basic Quantitative Correlation Inequality for Finite Product Domains} \label{subsec:our-talagrand}

Throughout this subsection, let $\Omega = \{0,1,\dots,m-1\}$ endowed with the natural ordering (though any $m$-element totally ordered set would do). We will consider monotone functions on $(\Omega^n, \pi^{\otimes})$; while our results hold in the more general setting of functions on $(\Omega^n, \otimes_{i=1}^n\pi_i)$, we stick to the setting of $L^2(\Omega^n, \pi^{\otimes n})$ for ease of exposition. 

In order to appeal to \Cref{thm:main-thm}, we must first show that the family of monotone (nondecreasing) functions on $\Omega^n$ is monotone compatible with the Bonami--Beckner noise operator (see \Cref{def:bonami-beckner-operator}). To this end, we define noise operators that act on each coordinate of the input:

\begin{definition}[coordinate-wise noise operators]
	Let $\T^i_\rho$ be the operator on functions $f : \Omega^n \to \R$ defined by
$$\T^i_\rho f(x) = \E_{\by \sim N_\rho(x_i)}\sbra{f(x_1, \ldots, \by, \ldots, x_n)},$$
and define $\T_{\rho_1, \ldots, \rho_n} f := \T^1_{\rho_1} \circ \T^2_{\rho_2} \circ \ldots
\circ \T_{\rho_n}^nf$.
\end{definition}

This is well-defined as the operators $\T_{\rho_i}^i$ and $\T_{\rho_j}^j$ commute. 


\begin{lemma}
\label{lem:noise-preserves-monotonicity-general}
Let $\Omega = \{0,1,\dots,m-1\}$ and let
$f: \Omega^n \to \R$ be a monotone function. Then $\T^i_\rho f: \Omega^n \to \R$
is a monotone function.
\end{lemma}

\begin{proof}
Suppose $x, y \in \Omega^n$ are such that $x_i \leq y_i$ for all $i\in [n]$. We
wish to show that
	$\T^i_\rho f(x) \leq \T^i_\rho f(y),$
	which is equivalent to showing 
$$\E_{\bz\sim N_\rho(x_i)}\sbra{f\pbra{x^{i\mapsto \bz}}}  \leq \E_{\bz\sim
N_\rho(y_i)}\sbra{f\pbra{y^{i\mapsto \bz}}}.$$
	
	Indeed, because of the monotonicity of $f$, via the natural coupling we have
	\begin{align*}
\E_{\bz\sim N_\rho(x_i)}\sbra{f\pbra{x^{i\mapsto \bz}}} &= \delta f(x) +
(1-\delta)\E_{\bz \sim \Omega^n}\sbra{f\pbra{x^{i\mapsto \bz}}} \\
& \leq \delta f(y) + (1-\delta)\E_{\bz \sim \Omega^n}\sbra{f\pbra{y^{i\mapsto
\bz}}} = \E_{\bz\sim N_\rho(y_i)}\sbra{f\pbra{y^{i\mapsto \bz}}}. \qedhere
	\end{align*}
\end{proof}

\begin{lemma}
\label{lem:noise-decreases-correlation-general}
Let $\Omega = \{0,1,\dots,m-1\}$ and let $f, g: \Omega^n \to
\R$ be monotone functions. Then $\langle \T_\rho f, g \rangle$ is nondecreasing in $\rho
\in [0, 1]$.
\end{lemma}

\begin{proof}	
	We have
$$\abra{ \T_{\rho_1, \ldots, \rho_n}f,g } = \abra{\T_{\rho, 1, \ldots, 1}f, T_{1, \rho_2,
\ldots, \rho_n}g} = \abra{\T^1_{\rho_1}f, h}$$
where $h := \T_{1, \rho_2, \ldots, \rho_n}g$. It follows from a repeated application
of \Cref{lem:noise-preserves-monotonicity-general} that $h$ is monotone.
Now, note that  
$$\abra{\T^1_{\rho_1}f, h} =
\widehat{f}\pbra{\bar{0}}\cdot\widehat{h}\pbra{\bar{0}} +
\sum_{\alpha_1 > 0}\rho_1\widehat{f}(\alpha)\widehat{h}(\alpha) +
\sum_{\substack{\bar{0}\neq \alpha \\ \alpha_1 = 0}}
\widehat{f}(\alpha)\widehat{h}(\alpha)$$
where $\bar{0} = (0, \ldots, 0)$. By Cheybshev's order inequality, we know that $\abra{ \T_1^1 f, h} \geq \abra{\T^1_0 f, h}
= \widehat{f}\pbra{\bar{0}}\cdot\widehat{h}\pbra{\bar{0}} +\sum_{\bar{0}\neq \alpha,\alpha_1 = 0}
\widehat{f}(\alpha)\widehat{h}(\alpha) $. From the
above expression, we have:
$$\frac{\partial}{\partial \rho_1} \abra{\T^1_{\rho_1}f, h} = \sum_{\alpha_1 >
0}\widehat{f}(\alpha)\widehat{h}(\alpha)$$
which must be nonnegative since $\abra{ \T_1^1 f, h} \geq \abra{\T^1_0 f, h}$,
and so we can conclude that $\abra{\T^1_{\rho_1}f, h}$ is nondecreasing in $\rho_1$.
The result then follows by repeating
this for each coordinate.
\end{proof}

Let $\Fmon \subseteq L^2(\Omega^n, \pi^{\otimes n})$ be the family of monotone functions $f : \Omega^n \to \R$. Then \Cref{lem:noise-decreases-correlation-general} shows that $\Fmon$ is monotone compatible with the Bonami--Beckner noise operator. We can now prove our Talagrand-analogue for monotone functions over $\Omega^n$:

\begin{theorem} \label{thm:our-talagrand}
	Let $\Omega = \{0,1,\dots,m-1\}^n$ and let $\Fmon \subseteq L^2(\Omega^n, \pi^{\otimes n})$ denote the family of monotone functions on $\Omega^n$ such that $\|f\| \leq 1$ for all $f\in\Fmon$. Then for $f, g \in \Fmon$, we have
	$$\E[fg] - \E[f]\E[g] \geq {\frac 1 C} \cdot \Phi\left(\sum_{\#\alpha = 1} \widehat{f}(\alpha)\widehat{g}(\alpha)\right)$$
	where recall from \Cref{eq:Phi} that $\Phi: [0,1] \to [0,1]$ is $\Phi(x) = \min\left\{x,\frac{x}{\log^2\left(1/x\right)}\right\}$ and $C>0$ is a universal constant.
\end{theorem}

\begin{proof}
	Consider the orthogonal decomposition 
	$$L^2(\Omega^n, \pi^{\otimes n}) = \bigoplus_{i=0}^n \calW_i$$
	where $\calW_i = \mathrm{span}\left\{\phi_\alpha : \#\alpha = i\right\}$; the orthogonality of this decomposition follows from the orthonormality of $(\phi_\alpha)_{\alpha\in\N^n_{<m}}$. Furthermore, this decomposition is a chaos decomposition with respect to the Bonami--Beckner operator $(\T_\rho)_{\rho\in[0,1])}$. It follows that the hypotheses of \Cref{thm:main-thm} hold for $\Fmon$ with $j^* = 1$, from which the result follows.
\end{proof}

\subsection{Interpreting \Cref{thm:our-talagrand} via the Efron--Stein Decomposition} \label{subsec:efron-stein-interpretation}

The goal of this subsubsection is to give an interpretation of
\Cref{thm:our-talagrand} which may shed some more light on it. The following notation will be useful.

\begin{definition} \label{def:fi}
Given $f: \Omega^n \to \R$ and $i \in [n]$, we define $f_i: \Omega \to \R$ to
be
\begin{align*}
f_i(x) &= \T^1_{0} \circ \cdots \circ \T^{i-1}_0 \circ \T^{i+1}_0  \circ \cdots
\circ \T^n_{0} f \\
&= \E_{\bx \sim \pi^{\otimes [n] \setminus
\{i\}}}[f(\bx_1,\dots,\bx_{i-1},x,\bx_{i+1},\dots,\bx_n)],
\end{align*}
the average value of $f$ over all ways of filling in the other $n-1$
coordinates and setting the $i$-th coordinate to $x$.
\end{definition}

By \Cref{lem:noise-preserves-monotonicity-general} we have the following:
\begin{fact} \label{fact:fi-monotone}
If $f: \Omega^n \to \R$ is monotone then so is $f_i$.
\end{fact}

Another easy fact is that the ``singleton'' Fourier coefficients of $f$
coincide with those of $f_i$:

\begin{claim} \label{claim:singleton}
Let $f: \Omega^n \to \R$ and let $i \in [n]$.  Then for any $j \in
\{0,\dots,m-1\}$, it holds that
\begin{equation} \label{eq:singleton}
\widehat{f}(j \cdot e_i) = \widehat{f_i}(j).
\end{equation}
\end{claim}

\begin{proof} We have
\begin{align*}
\widehat{f}(j \cdot e_i) &= \E_{\bx \sim \pi^{\otimes n}}[f(\bx) \phi_j(\bx_i)]\\
&= \E_{\bx_i \sim \pi}[\E_{\bx \sim \pi^{\otimes [n] \setminus
\{i\}}}[f(\bx_1,\dots,\bx_{i-1},\bx_i,\bx_{i+1},\dots,\bx_n)]\phi_j(\bx_i)]\\
&= \E_{\bx_i \sim \pi}[f_i(\bx_i) \phi_j(\bx_i) ] = \widehat{f_i}(j). \qedhere
\end{align*}
\end{proof}

We thus have that
\begin{align*}
\sum_{\#\alpha = 1} \widehat{f}(\alpha)\widehat{g}(\alpha) &=
\sum_{i=1}^n \sum_{j=1}^{m-1} \widehat{f}(j \cdot e_i) \widehat{g}(j \cdot e_j) = 
\sum_{i=1}^n \sum_{j=1}^{m-1} \widehat{f_i}(j) \widehat{g_i}(j).
\end{align*}
Fix an $i \in [n]$ and observe that if $f,g: \Omega^n \to \R$ are monotone,
then by Plancherel applied to the one-variable functions $f_i,g_i$, we have that
\begin{equation}~\label{eq:monotone-simple-1}
\sum_{j=1}^{m-1} \widehat{f_i}(j) \widehat{g_i}(j) = 
\E[f_i g_i] - \E[f_i] \E[g_i] = \E[f_i g_i] - \E[f] \E[g] \geq 0,
\end{equation}
where the non-negativity is because $f_i,g_i$ are monotone functions (by the
1-variable case of FKG, or equivalently by Chebyshev's order inequality). Finally, note that  the penultimate expression can be simplified in terms of the Efron-Stein decompositions of $f$ and $g$. In particular, we have 
\begin{eqnarray}
\E[f_i g_i] - \E[f] \E[g] &=& \E[(f^{=\{i\}} + \E[f]) (g^{=\{i\}} + \E[g])] - \E[f] \E[g] \nonumber\\
&=& \E[f^{=\{i\}} g^{=\{i\}}]  + \E[f^{=\{i\}}] \E[g] + 
\E[g^{=\{i\}}] \E[f]  \nonumber \\ 
&=& \E[f^{=\{i\}} g^{=\{i\}}]. \label{eq:monotone-simple-2}
\end{eqnarray}
The first equality uses that $f_i = f^{=\{i\}} + \E[f]$ and likewise $g_i = g^{=\{i\}} + \E[g]$. The third equality uses the fact $\E[f^{=\{i\}}] =\E[g^{=\{i\}}]=0$. Using \eqref{eq:monotone-simple-1} and \eqref{eq:monotone-simple-2}, 
we
thus have the following corollary of
\Cref{thm:our-talagrand}:
\begin{corollary}~\label{corr:interpretation-es}
Let $\Omega = \{0,1,\dots,m-1\}$ and suppose $f, g \in\Fmon$ where $\Fmon$ is as in \Cref{thm:our-talagrand}. Then $\E[f_i
g_i] - \E[f] \E[g] \geq 0$ for each $i \in [n]$ (hence $\sum_{\#\alpha = 1} \widehat{f}(\alpha)\widehat{g}(\alpha) = \sum_{i=1}^n (\E[f_i
g_i] - \E[f] \E[g]) \geq 0$), and
		\begin{align*}
\E[f\cdot g] - \E[f]\cdot\E[g] &\geq 
 {\frac 1 C} \cdot \Phi\left(\sum_{i = 1}^n (\E[f_i g_i]
- \E[f] \E[g])\right)\\
& = {\frac 1 C} \cdot
\Phi\left(\sum_{i = 1}^n \E[f^{=\{i\}} g^{=\{i\}}]\right)
		\end{align*}
where $C$ is a universal constant and $f^{=\{i\}}$ is the Efron-Stein
$\{i\}$-component of $f$.
\end{corollary}

We note that in addition to showing that the key quantity $\sum_{\#\alpha = 1} \widehat{f}(\alpha)\widehat{g}(\alpha)$ (the argument to $\Phi$ in \Cref{thm:our-talagrand}) is non-negative, the above discussion also lets us conclude that the key quantity $\sum_{\#\alpha = 1} \widehat{f}(\alpha)\widehat{g}(\alpha)$ is basis-independent. 

We conclude this subsection by showing that if  $\sum_{\#\alpha=1} \widehat{f}(\alpha) \widehat{g}(\alpha) =0$, then for every coordinate $i \in [n]$, at least one of $f,g$ is independent of the coordinate $i$. (This should be compared with the condition, discussed in \Cref{sec:qci} in the context of Talagrand's original quantitative correlation inequality \cite{Talagrand:96} for monotone Boolean functions over $\zo^n$, that $\Inf_i(f) \Inf_i(g) = 0$ for all $i$.)

\begin{lemma}~\label{lem:monotone-uncorrelated}
For $f, g \in \Fmon$, $\sum_{i=1}^n \E[f_i g_i] - \E[f_i] \E[g_i] =0$ if and only if there is a partition $(S, \overline{S})$ of $[n]$ such that $f$ is only dependent on the coordinates in $S$ and $g$ is only dependent on $\overline{S}$. 
\end{lemma}
\begin{proof}
We begin with the proof of the ``if direction". 
By definition of the Efron-Stein decomposition, it easily follows that $f^{=\{i\}}$ is identically $0$ for $i \not \in S$. Likewise, $g^{=\{i\}}$ is identically $0$ for $i \not \in \overline{S}$. Thus, we have 
\[
\sum_{i=1}^n \E[f_i g_i] - \E[f_i] \E[g_i] =\sum_{i=1}^n \E[f^{=\{i\}} g^{=\{i\}}]=0.
\]

To 
prove the ``only if" direction, 
 first recall that by Chebyshev's sum inequality, 
$\E[f_i g_i] - \E[f_i] \E[g_i] \ge 0$ for any $1 \le i \le n$, and since the underlying measure has full support, it is known that (crucially for us) equality holds iff at least one of $f_i$ or $g_i$ is the constant function. On the other hand, we know that $\sum_{i=1}^n \E[f_i g_i] - \E[f_i] \E[g_i]=0$. 
By \eqref{eq:monotone-simple-1},  we have  that for all $i$, 
$\E[f_i g_i] - \E[f_i] \E[g_i] = 0$. Thus, for every $i \in [n]$, we know that at least one of $f_i$ or $g_i$ is a constant function. 

Now, suppose $f_i$ is a constant function. Recalling that $f$ is monotone, it is easy to see that $f_i$ can only be a constant function if for all $x$ and $y$, if $x$ and $y$ differ only in the $i^{th}$ coordinate, then $f(x) = f(y)$. In other words, $f$ does not depend on the $i^{th}$ coordinate.  From this, it follows that there is  a  partition $(S, \overline{S})$ of $[n]$ such that $f$ is only dependent on the coordinates in $S$ and $g$ is only dependent on the coordinates in $\overline{S}$. 
\end{proof}

\subsection{Comparison with Keller's Quantitative Correlation Inequality for the $p$-biased Hypercube}


In this subsection we restrict our attention to the $p$-biased hypercube $\bits^n_p = (\bits^n, \pi_p^{\otimes n})$ where $\pi_p(-1) = p$ and $\pi_p(+1) = 1-p$.
In this setting our \Cref{thm:our-talagrand} generalizes Talagrand's inequality in two ways:   it holds  for real-valued monotone functions on $\bits^n$ that have 2-norm at most  1 (rather than just monotone Boolean functions), and it holds for any $p$ (as opposed to just $p=1/2$).  Keller \cite{kell08-prod-measure,Keller09} has earlier given a generalization of Talagrand's  inequality that holds for general $p$ and for real-valued monotone functions with $\infty$-norm at most 1:

\begin{theorem}[Theorem 7 of \cite{kell08-prod-measure}; see also \cite{Keller12} for a slightly weaker version]  \label{thm:keller-generalization}
Let $f, g \in L^2(\zo^n, \pi_p^{\otimes n})$ be monotone functions such that for all $x\in\bits^n$, we have $|f(x)|, |g(x)|\leq 1$. Then
$$\E[fg] - \E[f]\E[g] \geq {\frac 1 C} \cdot H(p)\cdot\Psi\left(\sum_{i=1}^n \widehat{f_p}(i)\widehat{g_p}(i)\right)$$
where $\widehat{f_p}(i)$ is the $p$-biased degree-1 Fourier coefficient on coordinate $i$, $\Psi : [0,1] \to [0,1]$ is given by $\Psi(x) = \frac{x}{\log(e/x)}$ as in \Cref{sec:qci}, $C>0$ is a universal constant, and $H: [0,1] \to [0,1]$ is the binary entropy function $H(x) = -x\log x - (1-x)\log(1-x).$
\end{theorem}

Comparing \Cref{thm:our-talagrand} to \Cref{thm:keller-generalization}, we see that the latter has an extra factor of $H(p)$, whereas the former shows that in fact no dependence on $p$ is necessary (but the former has an extra factor of $\frac{1}{\log\left(1/\sum_i \widehat{f_p}(i)\widehat{g_p}(i)\right)}$).  \Cref{thm:our-talagrand} can be significantly stronger than \Cref{thm:keller-generalization} in a range of natural settings because of these differences. In \Cref{ap:keller-example} we show that for every $\omega(1)/n \leq p \leq 1/2$, there is a pair of $\bits$-valued functions $f,g$ (depending on $p$) such that under the $p$-biased distribution (i) the quantity $\E[fg]-\E[f]\E[g]$ is at least an absolute constant independent of $n$ and $p$; (ii) the RHS of \Cref{thm:our-talagrand} is at least an absolute constant independent of $n$ and $p$; but (iii) the RHS of \Cref{thm:keller-generalization} is $\Theta(p \log(1/p)).$

\section{An Analogue of Talagrand's Correlation Inequality over $[-1, 1]^n$} 
\label{sec:talagrand-solid-cube}

In this section, we seek an analogue of~\Cref{thm:our-talagrand} when the domain is $[-1,1]^n$ endowed with the uniform measure $\calU$. The results of this subsection may be viewed as an attempt to answer a question posed by Keller \cite{Keller09}, who wrote ``It seems tempting to find a generalization of Talagrand's result to the continuous setting, but it is not clear what is the correct notion of influences in the continuous case that should
be used in such a generalization.'' In fact, we obtain two different quantitative correlation inequalities over $[-1, 1]^n$:  one using a \emph{replacement} noise operator (not unlike the Bonami--Beckner operator from \Cref{sec:application-talagrand}) which diagonalizes the basis of \emph{Legendre polynomials}, and another using the noise operator corresponding to \emph{reflected} Brownian motion on the interval $[0, 1]^n$ which diagonalizes the \emph{cosine basis}. 

\subsection{The Legendre Basis and Replacement Noise Operator}
\label{subsec:legendre}

Both the setup and the proof for the quantitative correlation inequality with respect to the Legendre basis are \emph{mutatis mutandis} analogous to the setup of and the proof of~\Cref{thm:our-talagrand}. Thus our exposition in this section is relatively succinct, with references to the relevant portions of~\Cref{subsec:our-talagrand}. 

Let $\calU$ be the uniform measure on the interval $[-1,1]$. It is a standard fact that $L^2([-1,1], \calU)$ is a separable Hilbert  space, and hence the space $L^2([-1,1], \calU)$ admits a countable orthonormal basis. In other words, there are functions $\{\phi^{c}_i\}_{i \ge 0}$ such that (i) each $\phi^{c}_i : [-1,1] \rightarrow \mathbb{R}$, and (ii) $\{\phi^{c}_i\}_{i \ge 0}$ is an orthonormal basis for $L^2([-1,1], U)$. (An explicit example of such a basis is the set of Legendre polynomials~\cite{Szego:39}.)  As in~\Cref{subsec:our-talagrand}, given the basis $\{\phi^c_i\}_{i \ge 0}$, we can obtain an orthonormal basis for $L^2([-1,1]^n, \calU^{\otimes n})$ via a standard product construction: Given a multi-index $\alpha \in \N^n$,  we define $\phi^c_\alpha \in L^2([-1,1]^n,\calU^{\otimes n})$ as
\[
\phi^{c}_\alpha(x) := \prod_{i=1}^n \phi^{c}_{\alpha_i}(x_i). 
\]
The collection $(\phi^{c}_\alpha)_{\alpha \in \N^n}$ is an orthonormal basis for $L^2([-1,1]^n, \calU^{\otimes n})$ (see Proposition~8.13 of \cite{od2014analysis}).
Similar to \eqref{eq:f-decomp-alpha} from Section~\ref{subsec:our-talagrand}, $f$ admits the following decomposition: 
\begin{equation}~\label{eq:f-decomp-alpha-uniform}
f = \sum_{\alpha \in \N^n} \widehat{f}(\alpha) \phi^{c}_{\alpha}.
\end{equation}
Analogous to~\Cref{def:bonami-beckner-operator}, we can define a ``replacement'' noise operator for $L^2([-1,1]^n, \calU^{\otimes n})$.

\begin{definition}[Noise operator]  
\label{def:bonami-beckner-operator-l2} 

For $\rho \in [0, 1]$ we define the \emph{noise operator for $L^2([-1,1]^n,\calU^{\otimes n})$} as the linear operator 
	$$\T^c_\rho f(x) := \E_{\by \sim N_\rho(x)}[f(\by)],
	$$
where ``$\by \sim N_\rho(x)$'' means that $\by \in \Omega^n$ is randomly chosen as follows: for each $i \in [n]$, with probability $\rho$ set $\by_i$ to be $x_i$ and with the remaining $1-\rho$ probability set $\by_i$ by independently making a draw from $U$.
\end{definition}

Similar to~\Cref{subsec:our-talagrand}, here again, it is easy to verify that  $\T^c_\rho f = \sum_{\alpha} \rho^{\#\alpha} \widehat{f}(\alpha) \phi^{c}_\alpha$. We now record the analogue of~\Cref{lem:noise-decreases-correlation-general}. The proof is identical to that 
of~\Cref{lem:noise-decreases-correlation-general} -- the only difference
is that we replace $\{\phi_\alpha\}$ by $\{\phi_\alpha^c\}$. 
\begin{lemma}
\label{lem:noise-decreases-correlation-general-continuous}
Let  $f, g: [-1,1]^n \to
\R$ be monotone functions. Then $\langle \T^c_\rho f, g \rangle$ is nondecreasing in $\rho
\in [0, 1]$.
\end{lemma}
Using the above lemma, we can obtain the following analogue of~\Cref{thm:our-talagrand} (with essentially the same proof). 
\begin{theorem} \label{thm:our-talagrand-cont}
	Let $\Fmon^c \subseteq L^2([-1,1]^n, \calU^{\otimes n})$ denote the family of monotone functions on $[-1,1]^n$ such that $\|f\| \leq 1$ for all $f\in\Fmon^c$. Then for $f, g \in \Fmon^c$, we have
	$$\E[fg] - \E[f]\E[g] \geq {\frac 1 C} \cdot \Phi\left(\sum_{\#\alpha = 1} \widehat{f}(\alpha)\widehat{g}(\alpha)\right)$$
	where recall from \Cref{eq:Phi} that $\Phi: [0,1] \to [0,1]$ is $\Phi(x) = \min\left\{x,\frac{x}{\log^2\left(1/x\right)}\right\}$ and $C>0$ is a universal constant.
\end{theorem}
Similar to~\Cref{corr:interpretation-es}, the consequence of~\Cref{thm:our-talagrand-cont} can also be interpreted in terms of the Efron-Stein decomposition.
\begin{corollary}~\label{corr:interpretation-es-cont}
Let $f, g \in\Fmon^c $ where $\Fmon^c$ is as in \Cref{thm:our-talagrand-cont}. Then $\E[f_i
g_i] - \E[f] \E[g] \geq 0$ for each $i \in [n]$ (hence $\sum_{\#\alpha = 1} \widehat{f}(\alpha)\widehat{g}(\alpha) = \sum_{i=1}^n (\E[f_i
g_i] - \E[f] \E[g]) \geq 0$), and
		\begin{align*}
\E[f\cdot g] - \E[f]\cdot\E[g] &\geq 
 {\frac 1 C} \cdot \Phi\left(\sum_{i = 1}^n (\E[f_i g_i]
- \E[f] \E[g])\right) = {\frac 1 C} \cdot
\Phi\left(\sum_{i = 1}^n \E[f^{=\{i\}} g^{=\{i\}}]\right)
		\end{align*}
where $C$ is a universal constant and $f^{=\{i\}}$ is the Efron-Stein
$\{i\}$-component of $f$.
\end{corollary}
Similar to~\Cref{lem:monotone-uncorrelated}, we also have the following lemma. 
\begin{lemma} \label{lem:monotone-uncorrelated-cont}
For $f, g \in \Fmon^c$, $\sum_{i=1}^n \E[f_i g_i] - \E[f_i] \E[g_i] =0$ if and only if there is a partition $(S, \overline{S})$ of $[n]$ such that $f$ is only dependent on the coordinates in $S$ and $g$ is only dependent on $\overline{S}$ up to measure zero sets. In other words, (i) there are functions $\overline{f}, \overline{g} \in \Fmon^c$ such that $\overline{f}$ (resp. $\overline{g}$) only depends on the coordinates in $S$ (resp. $\overline{S}$); (ii)$f$  (resp. $g$) is identical to $\overline{f}$ (resp. $\overline{g}$) except for a measure zero set. 
\end{lemma}
\begin{proof}
The proof of the ``if direction" is exactly as in~\Cref{lem:monotone-uncorrelated}. For the ``only if" part, as before, we can first infer that $\E[f_i g_i] - \E[f_i] \E[g_i]=0$ for all $1 \le i \le n$. Since  $f_i$ and $g_i$ are single variable monotone functions, the continuous version of Chebyshev's order inequality implies that $\E[f_i g_i] - \E[f_i] \E[g_i] \ge 0$, with equality holding if and only if one of $f_i$ or $g_i$ is a constant function (except possibly on a measure zero set). 

Next, fix an $i \in [n]$ (we take $i=1$ without loss of generality). At least one of $f_1$ or $g_1$ is a constant and again, without loss of generality, we assume that $f_1$ is a constant (up to a measure zero set). By the definition  of Efron-Stein decomposition, we have $f^{\{=1\}} = f_1 - \E[f]$. It follows that $f^{\{=1\}}$ is identically zero 
(except possibly on a measure zero set); so there is a set $\mathcal{A}^\ast \subseteq [-1,1]$  such that (i) $[-1,1] \setminus \mathcal{A}^\ast$ has measure zero; (ii) For all $x_1 \in \mathcal{A}^\ast$, $f^{\{=1\}} (x_1) = 0$. 
 
 Fix some $x_1^\ast \in \mathcal{A}^\ast$, and fix any $x_1 \in \mathcal{A}^\ast$ such that $x_1 >x_1^\ast$.  From the above, we have that
\begin{align}
\E_{\bx_2, \ldots, \bx_n} [f(x_1, \bx_2, \ldots, \bx_n) - f(x^\ast_1,\bx_2, \ldots, \bx'_n)] & = \left(  f^{\{=1\}} (x_1) + \E[f]\right) -
\left(  f^{\{=1\}} (x^\ast) + \E[f]\right) \label{fun:eq} \\
&= f^{\{=1\}} (x_1) - f^{\{=1\}} (x^\ast_1) = 0 - 0 = 0, \nonumber
\end{align}
where the first equality uses the fact that by definition of the Efron-Stein decomposition, for any $a \in [-1,1]$, we have
\[
\E_{\bx_2, \ldots, \bx_n} [f(a, \bx_2, \ldots, \bx_n)] = f^{\{=1\}} (a)  + \E[f]. 
\]
Also observe that as $f$ is monotone,  for any choice of $x_2, \ldots, x_n$, the term inside the expectation in the LHS of \eqref{fun:eq} is necessarily non-negative. Thus, \eqref{fun:eq} can be rewritten as  
\begin{equation}~\label{eq:fun1}
\E_{\bx_2, \ldots, \bx_n} [|f(x_1, \bx_2, \ldots, \bx_n) - f(x^\ast_1,\bx_2, \ldots, \bx'_n)|] =0.
\end{equation}
An identical argument show that if $x_1 \in \mathcal{A}^\ast$ is such that $x_1 <x_1^\ast$, we likewise have
\begin{equation}~\label{eq:fun2}
\E_{\bx_2, \ldots, \bx_n} [|f(x_1, \bx_2, \ldots, \bx_n) - f(x^\ast_1,\bx_2, \ldots, \bx_n)|] =0.
\end{equation}
Combining \eqref{eq:fun1}, \eqref{eq:fun2} and that $[-1,1]\setminus \mathcal{A}^\ast$ has measure zero, it follows that 
\[
\E_{\bx_1 , \bx_2, \ldots, \bx_n}  [|f(x_1, \bx_2, \ldots, \bx_n) - f(x'_1,\bx_2, \ldots, \bx_n)|] =0.
\]
Note that 
\[
\overline{f} (x)  := f(x'_1,x_2, \ldots, x_n)
\]
is a monotone function which does not depend on the coordinates in $\{1\}$ and $\E_{\bx} [|f(\bx) - \overline{f}(\bx)|] =0$. This argument can be iterated  to obtain a partition of coordinates into $(S,\overline{S})$ and functions $\overline{f}$ and $\overline{g}$ as promised in the claim. 
\end{proof}

\subsection{The Cosine Basis and the Reflected Heat Semigroup on $[0,1]^n$}
\label{subsec:talagrand-brownian}

In this subsection we give a different quantitative correlation inequality for monotone functions on the solid cube.  It will be convenient in this subsection for us to reparametrize the solid cube to $[0,1]^n$. We start by defining the \emph{cosine basis} for real-valued functions on $[0,1]^n$. 

Recall that any function $f \in \L^2([0,1],{\cal U})$ \ignore{\rnote{Tweaked this --- I think we don't want to say ``any function $f: [0,1] \to \R$'', people will go to town on us because of the kinds of weird things discussed in the history section of the Wikipedia page on Carleson's Theorem}} can be expressed as $f(x) = \hat{f}_0 + \sum_{k=1}^\infty \hat{f}_k\sqrt{2}\cos(\pi k x)$. To see this, extend $f: [0,1] \to \R$ to the function $g: [-1, +1] \to \R$ with $g(x) := f(|x|)$. The usual Fourier transform gives $g(x) = \hat{f}_0 + \sum_{k=1}^\infty \hat{f}_1\sqrt{2}\cos(\pi k x) + \sum_{k=1}^\infty \check{f}_1\sqrt{2}\sin(\pi k x)$. As $g$ is even and $\sin(\pi k x)$ is odd for all $k$, we have $\check{f}_k = 0$ for all $k \in \N$, and restricting $g$ to $[0,1]$ returns $f$. The cosine basis for $L^2\pbra{[0,1], \calU}$ can be extended via the usual product construction to obtain a basis for $L^2\pbra{[0,1]^n, \calU^{\otimes n}}$: given a multi-index $\alpha \in \N^n$,  we define $\varphi^c_\alpha : [0,1]^n \to \R$ as
\[
\varphi^{c}_\alpha(x) := \prod_{\alpha_i \neq 0} \sqrt{2}\cos(\pi \alpha_i x_i) 
\]
with $\varphi^c_{(0, \ldots, 0)} := 1$. It is easy to check that the collection $(\varphi^{c}_\alpha)_{\alpha \in \N^n}$ is an orthonormal basis for $L^2\pbra{[0,1]^n, \calU^{\otimes n}}$; therefore, every function $f \in \L^2([0,1]^n,{\cal U})$\ignore{$f: [0,1]^n \to \R$} admits a decomposition as
$$f = \sum_{\alpha \in \N^n} \hat{f}(\alpha) \varphi^{c}_{\alpha}.$$

We recall the definition of a Markov semigroup that diagonalizes the cosine basis. Let $p_t(\cdot, \cdot)$ denote the heat kernel on $\R^n$, that is
$$p_t(x, y) = \frac{1}{(4\pi t)^{n/2}} e^{-(\|x - y\|)^2/4t}.$$
In other words, for fixed $x\in\R^n$, $p_t(x, \cdot)$ is the density of a Gaussian random variable centered at $x$ and with covariance $2t I_n$ where $I_n$ is the $n\times n$ identity matrix. 

\begin{definition} \label{def:reflected-heat-semigroup}
	We define the \emph{reflected heat semigroup} $(\P_t^R)_{t\geq0}$ on $L^2([0,1], \calU)$ by
	$$\P_t^Rf(x) := \int_{0}^1 f(y) \pbra{\sum_{k \in \Z} p_t(x, 2k+y) + p_t(x, 2k - y)} dy.$$
\end{definition}

Note that this is a symmetric Markov semigroup with respect to the Lebesgue measure on $[0,1]$ \cite{bakry2013analysis}. The stochastic process underlying this semigroup is the \emph{reflected} Brownian motion $\bB_t^R$ given by
$$\bB_t^R = \begin{cases}
	\bB_t - \floor{\bB_t} & \floor{\bB_t} \text{ is even}\\ 
	1 - \bB_t + \floor{\bB_t} & \floor{\bB_t} \text{ is odd}\\
\end{cases}$$
where $\bB_t$ is the standard Brownian motion \ignore{{with diffusion coefficient $1/{\pi^2}$}}.
This lets us write
$$\P_t^R f(x) = \E\sbra{f\pbra{x + \sqrt{2}\bB_t^R}}.$$
It is readily verified (see e.g. Section~4.2.2 of \cite{pavliotis2014stochastic}) that the generator $\calL^R$ associated with this semigroup (see \Cref{def:generator}) is $\calL^R f = -f''$ with 
$$\{f \in C^\infty[0,1] : f'(0) = f'(1) = 0\} =: \calA \subseteq \dom{\calL^R}.$$ 
In order to avoid domain issues, we assume that all functions considered henceforth are in $\calA$, and that the expression $\P_t = e^{-t\calL}$ holds with the usual series expansion of $e^{-t\calL}$. An explicit calculation gives that the reflected heat semigroup acts on the cosine basis in the following way:

\begin{proposition} \label{prop:reflected-heat-basis}
\ignore{
\rnote{Is there a reference we can give for this - again, Section 4.2.2 of Pavliotis?  A more detailed reference? \shivam{I got this from a physicist's calculation based on what Joe does in his notes. I added this to the notes on reflected Brownian motion in the \texttt{papers/semigroup-stuff} folder.}}} Let $f : [0,1] \to \R$ such that $f\in\calA$. If $f = \sum_{k\in\N^n} \hat{f}_k\varphi_k^c$, then 
	$$\P_t^Rf = \sum_{k=0}^\infty e^{-k^2t} \hat{f}_k \varphi^c_k.$$
\end{proposition} 

Note that we can define $\P_{t}^R$ on $[0,1]^n$ via a tensorization process, and natural analogues of the above hold (as in \Cref{subsec:hermite-anal,subsec:efron-stein,subsec:legendre}). Towards establishing the monotone compatibility that we require for our approach, we first make the following simple observation:

\begin{proposition} \label{prop:reflected-heat-mon-preserving}
	If $f : [0,1]^n \to \R, f \in {\cal A}$ is monotone, then so is $\P_t^Rf$. 
\end{proposition}
\begin{proof}
We first consider the case of a one-dimensional function $f : [0,1] \to \R$, so the function $\P_t^R f$ is $\P_t^R f(x) = \E\sbra{f\pbra{x + \sqrt{2}\bB_t^R}}$. Suppose $x \leq y$: note that $f(x + \sqrt{2}\bB_{t}^R) \leq f(y + \sqrt{2}\bB_{t}^R)$ if the trajectories followed by the reflected Brownian motions starting at $x$ and $y$ do not intersect. If the trajectories do intersect, then we may couple the reflected Brownian motions from the time of intersection.  Hence using the strong Markov property, it follows that $\E\sbra{f\pbra{x + \sqrt{2}\bB_t^R}} \leq \E\sbra{f\pbra{y + \sqrt{2}\bB_t^R}}$.

For $n$-dimensional functions $f: [0,1]^n \to \R$, the result follows from the one-dimensional result as in the proof of \Cref{lem:noise-preserves-monotonicity-general} and the beginning of the proof of \Cref{lem:noise-decreases-correlation-general}.
\end{proof} 

 In order to prove the monotone compatibility of $\P_t^R$ with $\calA\cap\Fmon^c$, we will require Chebyshev's rearrangement inequality:

\begin{proposition} 	\label{prop:rearrangement-lemma}
	Let $f, g : [0, 1] \to \R$ be monotone functions, and let $\bx, \by \sim \calU$. Then $$\E_{\bx, \by}[f(\bx)g(\by)] \leq \E_{\bx}[f(\bx)g(\bx)].$$
\end{proposition}

\ignore{
\begin{proof}
	We will prove that the result holds for $f, g: [m] \to \R$ monotone and with respect to the uniform distribution on $[m]$, i.e. for $\bx, \by \sim \calU_{[m]}$. The desired result then follows by a limiting argument as $m \to \infty$.\snote{Minor technicality about the limiting argument---the convergence is in distribution, correct? I think we require $f, g$ to be continuous for this to be a true statement.} Note that by Chebyshev's rearrangement inequality, for every $\sigma \in S_m$, we have
	$$\E_{\bz\sim\calU_{[m]}}[f(\bz)g(\sigma(\bz))] \leq \E_{\bx\sim\calU_{[m]}}[f(\bx)g(\bx)].$$
	We claim that $(\bx, \by)$ can be expressed as a convex combination of $(\bz, \sigma(\bz))$ where $\sigma \in S_m$, which would complete the proof. Indeed, define $m\times m$ matrix $Q$ given by 
	$$Q(s, t) = \red{m} \cdot \Pr_{\bx, \by \sim \calU_{[m]}}[\bx = s, \by = t].$$
	Note that for all $s, t \in [m]$, we have $Q(s, t) \geq 0$ and that
	$$\sum_{t\in [m]} Q(s, t) = \sum_{s\in[m]} Q(s,t) = 1$$
	which implies that $Q$ is an element of the Birkhoff polytope. This implies that the matrix $\frac{1}{\red{m}}\cdot Q$ can be expressed as a convex combination of matrices of the form $\frac{1}{\red{m}}\cdot\mathrm{Perm}$ where $\mathrm{Perm}$ is a $m\times m$ permutation matrix. It follows that $(\bx, \by)$ for $\bx,\by\sim\calU_{[m]}$ can be expressed as a convex combination of $(\bz, \sigma(\bz))$ where $\bz\sim\calU_{[m]}$ and $\sigma\in S_m$, which completes the proof. 
\end{proof}
}

\begin{proposition} \label{prop:reflected-heat-mon-compatible}
Let $\Fmon^c$ be as in \Cref{thm:our-talagrand-cont}. Then $\calA \cap \Fmon^c $ is monotone compatible with $\P_t^R$.
\end{proposition}

\begin{proof}
	It suffices to show that $\la \P_t^R f, g\ra$ is decreasing in $t$ for one-dimensional monotone functions $f, g: [0, 1] \to \R$; the result for monotone-functions on the $n$-dimensional solid cube $[0,1]^n$ then follows a coordinate-wise application together with the self-adjointness of $\P_t$ (see \cite{bakry2013analysis}) as in the proof of \Cref{lem:noise-decreases-correlation-general}. Now, note that for any $\eps > 0$, we have
	$$\la \P_{t + \eps}^R f, g \ra = \la \P_{\eps}^R f, \P_t^Rg \ra = \la \P_{\eps}^R f, g' \ra$$
	where $g' := \P_t^Rg$ is a monotone function. As the Lebesgue measure is invariant with respect to $\P_t^R$ (see \cite{bakry2013analysis}), it follows that if $\bx\sim\calU$, then $\P_t^R\bx \sim \calU$. By \Cref{prop:rearrangement-lemma}, we then have that
	$$\la \P_{\eps}^R f, g' \ra \leq \la  f, g' \ra \qquad\text{and so}\qquad \la \P_{t + \eps}^R f, g\ra \leq \la \P_{t}^R f, g\ra$$
	from which it follows that $\la \P_t^R f, g\ra$ is decreasing in $t$.
\end{proof}

With this in hand, the following quantitative correlation inequality is immediate from \Cref{thm:main-thm} with $j^* = 1$.

\begin{theorem} \label{thm:solid-talagrand-brownian}
Let $\Fmon^c$ be as in \Cref{thm:our-talagrand-cont}. Then for $f, g \in \calA \cap \Fmon^c$, we have
	$$\E[fg] - \E[f]\E[g] \geq {\frac 1 C} \cdot \Phi\left(\sum_{\#\alpha = 1} \hat{f}(\alpha)\hat{g}(\alpha)\right)$$
	where recall from \Cref{eq:Phi} that $\Phi: [0,1] \to [0,1]$ is $\Phi(x) = \min\left\{x,\frac{x}{\log^2\left(1/x\right)}\right\}$ and $C>0$ is a universal constant.
\end{theorem}

\section*{Acknowledgments}
We thank Joe Neeman for clarifications concerning the reflected heat semigroup. 
A.D.~is supported by NSF grants CCF 1910534 and CCF 1926872.  S.N.~is supported
by NSF grants CCF-1563155 and by 
CCF-1763970.  R.A.S.~is supported by NSF grants CCF-1814873, IIS-1838154,
CCF-1563155, and by the Simons Collaboration on Algorithms and Geometry.  This
material is based upon work supported by the National Science Foundation under
grant numbers listed above. Any opinions, findings and conclusions or
recommendations expressed in this material are those of the author and do not
necessarily reflect the views of the National Science Foundation (NSF).

\ignore{

\newpage

\section*{Acknowledgments}

A.D.~is supported by NSF grants CCF 1910534 and CCF 1926872.  S.N.~is supported
by NSF grants CCF-1563155 and by 
CCF-1763970.  R.A.S.~is supported by NSF grants CCF-1814873, IIS-1838154,
CCF-1563155, and by the Simons Collaboration on Algorithms and Geometry.  This
material is based upon work supported by the National Science Foundation under
grant numbers listed above. Any opinions, findings and conclusions or
recommendations expressed in this material are those of the author and do not
necessarily reflect the views of the National Science Foundation (NSF).

}

\bibliography{correlation-writeup}{}
\bibliographystyle{alpha}

\appendix


\section{Proof of \Cref{claim:best-possible}} \label{ap:best-possible}

For $c \in \N$, let $T_c(x)$ denote the degree-$c$ Chebyshev polynomial of the first kind.
Define the univariate polynomial:
\[
a_d(t) := \frac{T_{\sqrt{d}}\left(t\left(1+\frac{3}{d}\right)\right)}{
T_{\sqrt{d}}\left(1+\frac{3}{d}\right)}
\]  
where $d$ is a parameter (a perfect square) that we will set later.  We make the following simple observations:

\begin{itemize}

\item $|a_d(t)| \leq 1$ for all $t \in [0,1]$,  and $a(1)=1.$
%

\item Let $d\geq 4$. For $t\in \sbra{0, 1-\frac{3}{d}}$, we have 
$a_d(t) \in \sbra{-\frac{1}{4}, \frac{1}{4}}$. This follows from the fact that $\left(1 -
\frac{3}{d}\right)\left(1 + \frac{3}{d}\right) < 1$, that $|T_{\sqrt{d}}(t)|
\leq 1$ for $|t| \leq 1$, and that the derivative $T'_{\sqrt{d}}(t)$ is at least $d$ for all $t \geq 1.$

\item The sum of the absolute values of the
coefficients of $a_d(t)$ is at most $2^{O\pbra{\sqrt{d}}}$. This is an easy consequence of standard coefficient bounds for Chebyshev polynomials (see e.g. Section 2.3.2 of \cite{mason2002chebyshev}).

\end{itemize}

For simplicity, assume $\log^2 M =4^k$ for some $k\in\N$. We define $b(t)$ as
$$b(t) := a_1(1-t)\cdot a_4(1-t) \cdot a_{16}(1-t) \cdots a_{\log^2M}(1-t).$$
Note that $b(t)$ is a polynomial of degree $\sqrt{1} + \sqrt{4} + \sqrt{16} + \ldots + \sqrt{\log^2 M} = \Theta(\log M)$, and that $|b(t)| \leq 1$ for all $t \in [0,1]$. It follows from the third item above that the sum of the absolute values of the coefficients of $b(t)$ is at most 
$$2^{O\pbra{\sqrt{1}} + O\pbra{\sqrt{4}} + \ldots + O\pbra{\sqrt{\log^2 M}}} = 2^{O(\log M)}.$$
Finally, we define $$p(t) := t\cdot b(t).$$ 
In order to upper bound $|p(t)|$ for $t \in [0,1]$, we first observe that if $t \leq {\frac 1 {4^k}}$ then we have $|p(t)| \leq {\frac 1 {4^k}} |b(t)| \leq {\frac 1 {4^k}} \leq {\frac 1 {\log^2 M}}$ as desired. Thus we may suppose that $t \in \sbra{\frac{1}{4^i}, \frac{1}{4^{i-1}}}$ for some $i \in \{1,\dots,k\}$; in particular, let $t = \frac{1}{4^i} + \delta$ for $\delta \in \sbra{0, \frac{3}{4^i}}$. Now, for each $j \geq i+1$, we have 
$$\abs{a_{4^j}(1-t)} \leq \frac{1}{4} \qquad\text{which implies that}\qquad |a_{4^{(i+1)}}(t)| \cdot |a_{4^{(i+2)}}(t)| \cdots |a_{4^k}(t)| \leq \frac{1}{4^{k-i}}.$$
As $t \leq \frac{1}{4^{i-1}}$, it follows that
$$|p(t)| = |t\cdot b(t)| \leq \frac{1}{4^{i-1}}\cdot\frac{1}{4^{k-i}} = \frac{1}{4^{k-1}} = \Theta\pbra{\frac{1}{\log^2 M}},$$
and \Cref{claim:best-possible} is proved.
It follows that \Cref{lem:main-lem} is tight up to constant factors.



\section{Proof of \Cref{lem:influence-basics}} \label{ap:influences-nonneg}

Recall that a function $f : \R^n \to \R_{\geq 0}$ is \emph{log-concave} if its domain is a convex set and it satisfies $f(\theta x + (1-\theta)y) \geq f(x)^\theta f(y)^{1-\theta}$ for all $x, y \in \mathrm{domain}(f)$ and $\theta \in (0,1)$. In particular, the $0/1$-indicator functions of convex sets are log-concave. We will require the following facts about log-concave functions:

\begin{fact}[Theorem 5.1, \cite{LV07}] 
\label{fact:marginal-log-concave}
	All marginals of a log-concave function are log-concave.
\end{fact}
The next fact is obvious from the definition of log-concave functions. 
\begin{fact}
\label{fact:unimodal-log-concave}
	A one-dimensional log-concave function is unimodal. 
\end{fact}

\begin{fact} [Lemma~4.7 of \cite{Vempala10a}]
\label{fact:1dlogconcave}
Let $g: \R \to \R^+$ be a  log-concave function such that \[\E_{\bx \sim \calN(0,1)}[\bx g(\bx)]=0.\] Then $\E[ \bx^2 g(\bx)] \le \E[g(\bx)],$ with equality if and only if $g$ is a constant function.

\ignore{(except possibly for 
a measure zero set).\cite{Vempala10a} does not state the exception of measure zero sets, but it is easily seen to be necessary; consider, for example, the log-concave function $g$ such that $g(x) = 1$ if $x \not =0$ and $g(0)=2$. \red{Rocco: How is this a log-concave function? Let $\theta = 1/2, x = 0, y=2$. Then $g(\theta x + (1-\theta)y)=g(1)=1$ but $g(x)^\theta g(y)^{1-\theta} = \sqrt{2}.$}}

\end{fact}

We will also require the following Brunn-Minkowski-type inequality over Gaussian space, as well as a characterization of the equality case in the following inequality. 

\begin{fact}[Ehrhard-Borell inequality, \cite{borell}]
	Let $A, B \subseteq \R^n$ be Borel sets, identified with their indicator functions. Then 
	\begin{equation} \label{eq:ehrhard}
		\Phi^{-1}\pbra{\gamma_n\pbra{\lambda A + (1-\lambda)B}} \geq \lambda\Phi^{-1}\pbra{\gamma_n\pbra{A}} + (1-\lambda)\Phi^{-1}\pbra{\gamma_n\pbra{B}} 
	\end{equation}
	where $\Phi : \R \to [0,1]$ denotes the cumulative distribution function of the standard, one-dimensional Gaussian distribution, $\gamma_n$ denotes the $n$-dimensional standard Gaussian measure, and $\lambda A +(1-\lambda)B := \{ \lambda x+ (1-\lambda)y : x\in A, y\in B \}$ is the \emph{Minkowski sum} of $\lambda A$ and $(1-\lambda)B$.
\end{fact}

\begin{fact}[Theorem 1.2 of \cite{rvh-equality}] \label{fact:rvh-equality-ehrhard}
	We have equality in the Ehrhard-Borell inequality (\Cref{eq:ehrhard}) if and only if either
	\begin{itemize}
		\item $A$ and $B$ are parallel halfspaces, i.e. we have
		\[A = \{ x : \abra{a, x} + b_1 \geq 0 \} \qquad\text{and}\qquad B = \{ x : \abra{a, x} + b_2 \geq 0 \}\] 
		for some $a \in \R^n$, and $b_1, b_2 \in \R$; or

		\item $A$ and $B$ are convex sets with $A = B$.
	\end{itemize}
\end{fact}

\begin{proof}[Proof of \Cref{lem:influence-basics}]

	Without loss of generality, let $v = e_1$. We have
	\begin{align}
		\Inf_{e_1}[K] = -\wt{K}(2e_1) &= \E_{\bx\sim\calN(0,1)^n}\sbra{-K(\bx)h_2(\bx_1)}\nonumber\\
		&= \E_{\bx_1\sim\calN(0,1)}\sbra{-\pbra{\underbrace{\E_{(\bx_2, \ldots, \bx_n)\sim\calN(0,1)^{n-1}}\sbra{K(\bx_1, \ldots, \bx_n)}}_{=: g(\bx_1)}}h_2(\bx_1)} \label{num:cucumber}\\
		& = -\wt{g}(2) \nonumber
	\end{align}
	where $g$ is a univariate function. From \Cref{fact:marginal-log-concave}, it follows that $g$ is log-concave. Furthermore, $g$ is centrally symmetric as $K$ is centrally symmetric, so $\E_{\bx_1 \sim \calN(0,1)}[\bx_1 g(\bx_1)]=0$. Hence, using the fact that $h_2(x_1) = (x_1^2-1)/\sqrt{2}$, we get that
	\[
	\Inf_{e_1}[K] = -\E_{\bx_1 \sim \calN(0,1)}[h_2(\bx_1) g(\bx_1)] = {\frac 1 {\sqrt{2}}} \cdot \E_{\bx_1 \sim \calN(0,1)}\big[g(\bx_1)(1-\bx_1^2)  \big] \ge 0, 
	\]
where the inequality is by~\Cref{fact:1dlogconcave}.	

Next, we move to  the characterization of $\Inf_{e_1}[K]=0$. Note that if $K(x) = K(y)$ whenever $x_{e_1^\perp} = y_{e_1^\perp}$ (i.e. $K(x) = K(y)$ whenever $(x_2, \ldots, x_n)= (y_2, \ldots, y_2)$), then the function $K$ does not depend on the variable $x_1$. This lets us re-express \eqref{num:cucumber} as 
	\[
	\E_{\bx_1\sim\calN(0,1)}[-h_2(\bx_1)] \E_{(\bx_2, \ldots, \bx_n)\sim\calN(0,1)^{n-1}}[K(\cdot, \bx_2, \ldots, \bx_n)]. 
	\]
	As the first term in the above product is zero, we conclude that $\Inf_{e_1}[K]=0$. 
	
	To see the reverse direction, suppose $\Inf_{e_1}[K]=0$. From \eqref{num:cucumber} and~\Cref{fact:1dlogconcave}, it follows that $g(\cdot)$ is a constant function. Now, for any $\alpha \in \mathbb{R}$, define $K_{\alpha} \subset \R^{n-1}$ as follows: \[
	K_\alpha :=\{ (x_2, \ldots, x_n) : (\alpha, x_2, \ldots, x_n) \in K\}.\] 
	Thus, $K_{\alpha}$ is the convex set obtained by intersecting $K$ with the affine plane $\{x \in \R^n: x_1=\alpha\}$. 
Observe that $g(\alpha)$ is the $(n-1)$-dimensional Gaussian measure of $K_{\alpha}$. Let $K_\alpha^* := {{\frac 1 2}}(K_{\alpha} + K_{-\alpha})$. Note that $K_\alpha^* \subseteq \R^{n-1}$ is a centrally symmetric, convex set, and that $K_\alpha^* \subseteq K_0$ because of convexity. By the Ehrhard-Borell inequality, we have 
\[ 
\Phi^{-1}\pbra{\gamma_{n-1}\pbra{\frac{K_\alpha + K_{-\alpha}}{2}}} \geq \frac{1}{2}\Phi^{-1}\pbra{\gamma_{n-1}(K_\alpha)} + \frac{1}{2}\Phi^{-1}\pbra{\gamma_{n-1}(K_{-\alpha)}}.
\]
However, $\gamma_{n-1}(K_\alpha) = \gamma_{n-1}(K_{-\alpha})$ because $K$ is centrally symmetric, so it follows that $\gamma_{n-1}(K_\alpha^*) \geq \gamma_{n-1}K_\alpha$. From our earlier observation that $K_\alpha^* \subseteq K_0$ and that $g(\cdot)$ is constant (which implies $\gamma_{n-1}(K_\alpha) = \gamma_{n-1}(K_0)$), it follows that $K_0 = K_\alpha^*$ up to a set of measure zero. In other words, we have equality in the application of the Ehrhard-Borell inequality above, and so by \Cref{fact:rvh-equality-ehrhard}, we must have $K_\alpha = K_{-\alpha}$ (that $K_\alpha$ and $K_{-\alpha}$ cannot be parallel halfspaces is clear). Consequently, {up to a set of measure zero, we have that}
\[K_0 = K_\alpha^* = \frac{K_\alpha + K_{-\alpha}}{2} = \frac{K_\alpha + K_{\alpha}}{2} = K_\alpha.\] 
As this is true for all $\alpha \in \R$, it follows that {up to a set of measure zero,} $K(x) = K(y)$ if $(x_2, \ldots, x_n) = (y_2, \ldots, y_n)$ (where we used the fact that $K(x) = K_{x_1}(x_2, \ldots, x_n)$). 
\end{proof}
	
\ignore{	
	Next, note that $g_1$ is a symmetric log-concave function which is a constant up to a measure zero set.
	It follows that for some $\nu>0$, 
	\begin{equation}~\label{eq:Gaussian-volume} \forall x_1 \in \mathbb{R} \setminus \{0\}, \quad g_1(x_1) = \nu. 
	\end{equation}  This uses~\Cref{fact:unimodal-log-concave} and the fact that $K$ has positive Gaussian volume (the desired statement is easily seen to hold if the Gaussian volume of $K$ is zero). By standard tail bounds on the chi-square distribution with $n-1$ degrees of freedom, it must be the case that 
	\begin{equation}~\label{eq:far-point}
	\forall x_1 \in \mathbb{R} \setminus \{0\}, \quad \exists (x'_2, \ldots, x'_n): \quad (x_1, x'_2, \ldots, x'_n) \in K \ \textrm{and} \ \Vert (x'_2, \ldots, x'_n) \Vert \le \sqrt{n \log(1/\nu)}. 
	\end{equation}
(It will be clear from the rest of the argument that the specific functional form $\sqrt{n \log(1/\nu)}$ is not important; any finite function of $n$ and $1/\nu$ would do.)	
 We now use this to show  that for any $\alpha, \beta \in \R$, we have $\mathrm{int}(K_{\alpha}) = \mathrm{int}(K_{\beta})$ where $\mathrm{int}(K_{\alpha})$ denotes the interior of the set $K_{\alpha}$. 
	
Let us assume that $\beta \le \alpha$ (the other direction is similar). 
Fix any point $y = (\beta, y_2, \ldots, y_n)$ where $(y_2, \ldots, y_n) \in K_{\beta}$.
Given any \red{$x_1 \geq \alpha$} \rnote{Was $x_1 > \beta$}, let $x'= (x_1, x'_2, \ldots, x'_n) \in K$ be such that $(x'_2, \ldots, x'_n)$ has the property guaranteed by \eqref{eq:far-point}, i.e.~$\Vert (x'_2, \ldots, x'_n) \Vert \le \sqrt{n \log(1/\nu)}$. Let us define the point $z=z(x_1)$ which is the following convex combination of $x'$ and $y$:\rnote{Was
\[
z = {\frac \alpha {x_1 - \beta}} x' +
{\frac {x_1 - \beta -\alpha} {x_1 - \beta}} y.
\]
}
\[
z = \red{{\frac {\alpha - \beta} {x_1 - \beta}} x' +
{\frac {x_1  -\alpha} {x_1 - \beta}} y.}
\]

Since $K$ is convex and $x',y \in K$ we have that $z$ also is in $K$. Observe that the first coordinate $z_1$ of $z$ is equal to $\alpha$, and hence $z \in K_\alpha.$

Now it is easy to see that as $x_1 \rightarrow +\infty$, the distance $\Vert z -  (\alpha, y_2, \ldots, y_n)\Vert \rightarrow 0$.   In other words, for every point $s \in K_{\beta}$ and for every $\epsilon >0$, there is a point
	$t \in K_{\alpha}$ such that $\Vert s -t \Vert_2 \le \epsilon$.  
	As $K$ is convex, it follows that $\mathrm{int}(K_{\alpha})$ contains $\mathrm{int}(K_{\beta})$. A symmetric argument (in which $x_1 \to -\infty$) implies that $\mathrm{int}(K_{\alpha})\subseteq \mathrm{int}(K_{\beta})$. Thus, we conclude that for all $\alpha, \beta$, $\mathrm{int}(K_\alpha) = \mathrm{int}(K_\beta)$. This implies that up to a measure zero set, $K(x) = K(y)$ whenever $x_{e_1^{\perp}} = y_{e_1^{\perp}}$, and establishes the proposition.
	
}

	\ignore{
	Let us assume that $\beta \le \alpha$ (the other direction is similar). Choose any $x_1$ and let $x'= (x_1, x'_2, \ldots, x'_n) \in K$ where $(x'_2, \ldots, x'_n)$ has the property guaranteed by \eqref{eq:far-point}. Now, consider a point $y = (\beta, y_2, \ldots, y_n)$ where $(y_2, \ldots, y_n) \in K_{\beta}$. Let $z$ be a convex combination of the point $y$ and $x'$ such that its first coordinate is $\alpha$. It is easy to see that as $x_1 \rightarrow \infty$, $\Vert z -  (\alpha, y_2, \ldots, y_n)\Vert \rightarrow 0$.   In other words, for every point $s \in K_{\beta}$ and all $\epsilon >0$, there is a point
	$t \in K_{\alpha}$ such that $\Vert s -t \Vert_2 \le \epsilon$.  
	As $K$ is convex, it follows that $\mathrm{int}(K_{\alpha})$ contains $\mathrm{int}(K_{\beta})$. A symmetric argument implies that $\mathrm{int}(K_{\alpha})\subseteq \mathrm{int}(K_{\beta})$. Thus, we conclude that for all $\alpha, \beta$, $\mathrm{int}(K_\alpha) = \mathrm{int}(K_\beta)$. This implies that up to a measure zero set, for all $x, y$, $K(x) = K(y)$ whenever $x_{e_1^{\perp}} = y_{e_1^{\perp}}$, and establishes the proposition.
}




\ignore{

\subsection{Proof of \Cref{lem:integration-by-parts}}
\label{ap:int-by-parts}

We will require the following well-known identity:

\begin{lemma}[Stein's identity, \cite{Stein72}]
\label{lem:steins-identity}	
If $X \sim \calN(\mu, \sigma^2)$ and $g$ is a function  such that the two
expectations $\E[(X - \mu)g(X)], \E[g'(X)]$ exist, then $$\E[(X - \mu)g(X)]
= \sigma^2\E[g'(X)].$$
\end{lemma}

\begin{proof}[Proof of \Cref{lem:integration-by-parts}]
Consider the function
$$A_s(x, z) = f(x)\cdot g\pbra{\sqrt{s}x + \sqrt{1-s}z}.$$
We have \rnote{changed some $t$'s to $\rho$'s below}
$$\frac{\partial}{\partial \red{\rho}} \abra{\U_\rho f, g} = \frac{\partial}{\partial s}
\E_{\bx, \bz \sim \calN(0,1)^n} \sbra{A_s(\bx, \bz)} \cdot \frac{\partial
s}{\partial \red{\rho}} = \E_{\bx, \bz \sim \calN(0,1)^n}\sbra{\frac{\partial}{\partial
s} A_s(\bx, \bz)}\cdot\frac{\partial s}{\partial \red{\rho}}$$
with $s = \rho^2$. By the chain rule, we get
\begin{align*}
\frac{\partial}{\partial s} A_s(x, z) &= \frac{1}{2} f(x) \sum_{j=1}^n
\pbra{\partial_j g\pbra{\sqrt{s}x + \sqrt{1-s}z}\pbra{\frac{x_j}{\sqrt{s}} -
\frac{z_j}{\sqrt{1-s}}}}\\
&= \frac{1}{2} \sum_{j=1}^n f(x)\partial_j g\pbra{\sqrt{s}x +
\sqrt{1-s}z}\pbra{\frac{x_j}{\sqrt{s}} - \frac{z_j}{\sqrt{1-s}}}.
\end{align*}

Let $w = \sqrt{s}x + \sqrt{1-s}z$ and $u = \sqrt{1-s}x - \sqrt{s}z$, and let
$\bw, \bu$ be the corresponding random variables with $\bx, \bz \sim \calN(0,1)^n$.
Note that $\bw, \bu \sim \calN(0, 1)^n$ and are uncorrelated, hence independent.
As $x = \sqrt{1-s}u + \sqrt{s}w$, we have:
\begin{align*}
\frac{\partial}{\partial s} A_s(x, z) &= \frac{1}{2}\sum_{j=1}^n
f\pbra{\sqrt{1-s}u + \sqrt{s}w}\cdot\partial_j
g(w)\cdot\frac{u_j}{\sqrt{s(1-s)}}\\
&=\frac{1}{2\sqrt{s(1-s)}}\sum_{j=1}^n \partial_j g(w)\cdot f\pbra{\sqrt{1-s}u
+ \sqrt{s}w}\cdot u_j
\end{align*} 
and taking expectations, we get:
\begin{align*}
\E_{\bx, \bz}\sbra{\frac{\partial}{\partial s} A_s(\bx, \bz)} &=
\frac{1}{2\sqrt{s(1-s)}}\sum_{j=1}^n \E_{\bw, \bu}\sbra{\partial_j g(\bw)\cdot
f\pbra{\sqrt{1-s}\bu + \sqrt{s}\bw}\cdot \bu_j}\\
&= \frac{1}{2\sqrt{s(1-s)}}\sum_{j=1}^n \E_{\substack{\bw\\
\bu\backslash\bu_j}}\sbra{\partial_j g(\bw)\cdot
\E_{\bu_j}\sbra{f\pbra{\sqrt{1-s}\bu + \sqrt{s}\bw}\cdot \bu_j}}\\
&= \frac{1}{2\sqrt{s(1-s)}}\sum_{j=1}^n \E_{\substack{\bw\\
\bu\backslash\bu_j}}\sbra{\partial_j g(\bw)\cdot\E_{\bu_j}\sbra{\partial_j
f\pbra{\sqrt{1-s}\bu + \sqrt{s}\bw}\cdot\sqrt{1-s}}}
\end{align*}
where the last equality follows from \Cref{lem:steins-identity}. As $\bw,
\bu \sim \calN(0,1)^n$, we have:
$$\E_{\bx, \bz}\sbra{\frac{\partial}{\partial s} A_s(\bx, \bz)} =
\frac{1}{2\sqrt{s}} \sum_{j=1}^n\E_{\bx, \bz}\sbra{\partial_j
f(\bx)\cdot\partial_j g\pbra{\sqrt{s}\bx + \sqrt{1-s}\bz}}.$$
By our change of variables, we have:
\begin{align*}
\frac{\partial}{\partial \rho} \abra{\U_\rho f, g} &= \frac{\partial}{\partial s}
\E_{\bx, \bz \sim \calN(0,1)^n} \sbra{A_s(\bx, \bz)} \cdot \frac{\partial
s}{\partial \rho} \\
&= \frac{1}{2\rho}\sum_{j=1}^n \E_{\bx, \bz}\sbra{\partial_j f(\bx)\cdot\partial_j
g\pbra{\rho\bx + \sqrt{1-\rho^2}\bz}} \cdot 2\rho \\
&= \sum_{i=1}^n \E_{\bx, \by\sim N_\rho(\bx)}\pbra{\partial_i f(\bx)\partial_i g(\by)}
\end{align*}
completing the proof of \Cref{lem:integration-by-parts}. 
\end{proof}
}


\section{Comparison of \Cref{thm:our-talagrand} and \Cref{thm:keller-generalization} }
\label{ap:keller-example}

Let $\omega(1)/n \leq p \leq 1/2$.  Observe that under $\bn_p$ we have $\E[\bx_1 + \cdots + \bx_n] = n(1-2p)$.
We define $f: \bn \to \bits$ to be the ``$p$-biased analogue of the majority function,'' i.e.
\[
f(x) := \sign(x_1 + \cdots + x_n - n(1-2p)),
\]
and we take $g=f.$

Since (as is well known) the median of the Binomial distribution $\mathrm{Bin}(n,p)$ differs from the mean by at most 1, it follows (using the Littlewood-Offord anticoncentration inequality described below) that $\E[f] = o(1)$, and hence we have (i):
$\E[fg] - \E[f]\E[g] \geq 1 - o(1).$ To establish (ii) and (iii) it remains only to show that
for any fixed $i \in [n]$ we have that the $p$-biased degree-1 Fourier coefficient $\widehat{f_p}(i)$ is at least $\Omega(1/\sqrt{n})$, or equivalently, that
$\widehat{f_p}(1) + \cdots + \widehat{f_p}(n) = \Omega(\sqrt{n}).$ 
To see this, we observe that this sum of degree-1 Fourier coefficients is
\begin{align}
\sum_{i=1}^n \widehat{f_p}(i) &=
\E\left[f(\bx) \cdot \sum_{i=1}^n {\frac {\bx_i - (1-2p)}{2\sqrt{p(1-p)}}}\right] = {\frac 1 {2\sqrt{p(1-p)}}} \Ex\left[\left|\left(\sum_{i=1}^n \bx_i\right) - n(1-2p) \right|\right]. \label{eq:L11}
\end{align}
We now recall the Littlewood-Offord anticoncentration inequality for the $p$-biased Boolean hypercube (see e.g.~Theorem~5 of \cite{de2017inverse} or \cite{AGKW:09}). Specialized to our context, this says that for any real interval $I$ of length at least 1, it holds that $\Pr\left[\sum_{i=1}^n \bx_i \in I\right] \leq O(|I|)/\sqrt{np(1-p)}.$ Taking $I$ to be the interval of length $c \sqrt{np(1-p)}$ centered at $n(1-2p)$ for a suitably small positive constant $c$, it holds that 
\[
\Pr\left[\left|\left(\sum_{i=1}^n \bx_i\right) - n(1-2p) \right| \geq c\sqrt{np(1-p)}\right] \geq {\frac 1 2}.\]
Consequently
\[
\Ex\left[\left|\left(\sum_{i=1}^n \bx_i\right) - n(1-2p) \right|\right] \geq  {\frac {c\sqrt{np(1-p)}}{2}},
\]
which together with \Cref{eq:L11} gives that $\sum_{i=1}^n \widehat{f_p}(i) \geq c\sqrt{n}/4$ as desired.

\end{document}